\documentclass[reqno]{amsproc}
\usepackage{mathrsfs}
\usepackage{bbm}
\usepackage{mathrsfs,amscd,amssymb,amsthm,amsmath,bm,graphicx,psfrag,subfigure,xypic}
\input epsf
\newtheorem{theorem}{Theorem}[section]
\newtheorem{lemma}[theorem]{Lemma}

\newtheorem{thm}[theorem]{Theorem}
\newtheorem{prop}[theorem]{Proposition}
\newtheorem{cor}[theorem]{Corollary}

\theoremstyle{definition}

\theoremstyle{remark}

\newtheorem{exmp}[theorem]{Example}

\numberwithin{equation}{section}

\usepackage{bm}

\def\mod{{\rm mod\;}}
\def\Mod{{\rm Mod}}
\def\Tor{{\rm Tor}}

\def\supp{\mbox{supp\,}}

\def\rk{\mbox{\rm\scriptsize rank}}

\def\FL{{\textsc{fl}}}

\def\sp{\hspace{1ex}}
\def\sp{\hspace{0.3cm}}

\begin{document}

\title[Orientations, lattice polytopes, and group arrangements]
{Orientations, lattice polytopes, and group arrangements II: Modular and
integral flow polynomials of graphs}

\author{Beifang Chen}
\address{Department of Mathematics,
Hong Kong University of Science and Technology,
Clear Water Bay, Kowloon, Hong Kong}
\curraddr{}

\email{mabfchen@ust.hk}
\thanks{Research of the first author is supported by RGC Competitive Earmarked Research Grants
600703, 600506, and 600608}

\author{Richard Stanley}
\address{Department of Mathematics,
Massachusetts Institute of Technology, Cambridge, MA 02139, USA}
\email{rstan@math.mit.edu}
\thanks{}

\date{28 October 2005}



\begin{abstract}
We study modular and integral flow polynomials of graphs by means of
subgroup arrangements and lattice polytopes. We introduce an
Eulerian equivalence relation on orientations, flow arrangements,
and flow polytopes; and we apply the theory of Ehrhart polynomials
to obtain properties of modular and integral flow polynomials. The
emphasis is on the geometrical treatment through subgroup
arrangements and Ehrhart polynomials. Such viewpoint leads to a
reciprocity law for the modular flow polynomial, which gives rise to
an interpretation on the values of the modular flow polynomial at
negative integers, and answers a question by Beck and Zaslavsky.
\end{abstract}

\maketitle

\section{Introduction}

The flow polynomial $\varphi(G,t)$ of a graph $G$ was introduced by
Tutte \cite{Tutte2} as a conceptual dual to the chromatic (or
tension) polynomial of $G$. When $G$ is a planar graph,
$\varphi(G,t)$ is essentially the chromatic polynomial $\chi(G^*,t)$
of the dual graph $G^*$ in the sense that
$\chi(G^*,t)=t^{c(G)}\varphi(G,t)$, where $c(G)$ is the number of
connected components of $G$. The historic Four-Color Conjecture of
the time was made by Tutte into the Five-Flow Conjecture: any
bridgeless graph admits a nowhere-zero integer 5-flow. Both
conjectures are still open and stimulate studies on chromatic and
flow polynomials. In a seminal paper \cite{Rota1}, Rota introduced
characteristic polynomial for posets and observed that
$\varphi(G,t)$ is the characteristic polynomial of the circuit
lattice of $G$. Greene and Zaslavsky \cite{Greene-Zaslvsky1} made
Rota's observation transparent between the flow polynomial and the
characteristic polynomial by using hyperplane arrangements. As a
special case of Zaslavsky's formula \cite{Zaslavsky1}, the absolute
value $|\varphi(G,-1)|$ counts the number of totally cyclic
orientations of $G$, which is a dual analog of Stanley's result on
chromatic polynomials: $|\chi(G,-1)|$ counts the number of acyclic
orientations of $G$. However, Stanley's result \cite{Stanley1}
includes an interpretation of the values of $\chi(G,t)$ at negative
integers, known as the Reciprocity Law of chromatic polynomials.
More recently, Kochol \cite{Kochol1} showed that the number of
integer-valued $q$-flows is a polynomial function of $q$ and
introduced the integral flow polynomial $\varphi_{\mathbbm z}(G,t)$;
Beck and Zaslavsky \cite{Beck-Zaslavsky1} studied the modular and
integral flow polynomials for graphs and signed graphs, using
Ehrhart polynomials of lattice polytopes. The present paper, as a
continuation of \cite{OLS-I}, is to associate flow group
arrangements with graphs, and to obtain a clear picture of the
relation between the integral flow polynomial and the modular flow
polynomial. The byproduct of this association is a generalization of
the Reciprocity Law of chromatic and tension polynomials to modular
and integral flow polynomials, and the interpretation of the values
of the modular and integral flow polynomials at zero and negative
integers. The geometric method of our exposition may be modified to
obtain analogous results on Tutte polynomials.

Let $G=(V,E)$ be a finite graph with possible loops and multiple
edges. We write $V=V(G)$, $E=E(G)$. For each subset $X\subseteq E$,
denote by $\langle X\rangle$ the induced subgraph $(V,X)$. An {\em
orientation} on $G$ is a (multivalued) function $\varepsilon:
V\times E\rightarrow\{-1,0,1\}$ such that (i) $\varepsilon(v,e)$ has
the ordered double-value $\pm1$ or $\mp1$ if $e$ is a loop at a
vertex $v$ and has a single-value otherwise, (ii)
$\varepsilon(v,e)=0$ if $v$ is not an end-vertex of $e$, and (iii)
$\varepsilon(u,e)\varepsilon(v,e)=-1$ if $e$ has two distinct
end-vertices $u,v$. Pictorially, if $e$ is a non-loop edge with
distinct end-vertices $u,v$, then
$\varepsilon(u,e)=-\varepsilon(v,e)=1$ (or
$\varepsilon(v,e)=-\varepsilon(u,e)=1$), and it means that $e$ is
assigned an arrow from $u$ to $v$, which contributes exactly one
out-degree at $u$ and one in-degree at $v$. If $e$ is a loop at a
vertex $v$, then $\varepsilon(v,e)=\pm1$ or $\mp1$, and it means
that the loop $e$ is assigned an arrow, pointing away and to the
vertex $v$, which contributes exactly one out-degree and one
in-degree at $v$. We assume $-(\pm1)=\mp1$, $-(\mp1)=\pm1$. A graph
$G$ together with an orientation $\varepsilon$ is called a {\em
digraph}, denoted $(G,\varepsilon)$. A digraph is said to be {\em
directed Eulerian} if its in-degree equals its out-degree at every
vertex.

Let $(G,\varepsilon)$ be a digraph throughout the whole paper.
Associated with $(G,\varepsilon)$ is the {\em incidence matrix}
${\bm M}={\bm M}(G):=[{\bm m}_{v,e}]_{V\times E}$, where ${\bm
m}_{v,e}=0$ if the edge $e$ is a loop and ${\bm
m}_{v,e}=\varepsilon(v,e)$ if $e$ is not a loop. Let $A$ be an
abelian group. A {\em flow} of $(G,\varepsilon)$ with values in $A$,
or an {\em $A$-flow}, is a function $f:E\rightarrow A$, satisfying
the {\em Conservation Law:}
\begin{equation}\label{Balance-Equation}
\sum_{e\in E}{\bm m}_{v,e}(v,e)f(e)=0\sp \mbox{or}\sp \sum_{e\in
E}\varepsilon(v,e)f(e)=0,\sp v\in V,
\end{equation}
where $\varepsilon(v,e)$ is counted twice in the second sum as $-1$
and $1$ if $e$ is a loop at its unique end-vertex $v$. A flow $f$ is
said to be {\em nowhere-zero} if $f(e)\neq 0$ for all $e\in E$. We
denote by $F(G,\varepsilon;A)$ the abelian group of all $A$-flows of
$(G,\varepsilon)$. The {\em flow arrangement} of $(G,\varepsilon)$
with the abelian group $A$ is the group arrangement ${\mathcal
A}_\FL(G,\varepsilon;A)$ of $F(G,\varepsilon;A)$, consisting of the
subgroups
\begin{equation}\label{Edge-Graph-Arrangement}
F_e:=\{f\in F(G,\varepsilon;A)\:|\: f(e)=0\}, \sp e\in E.
\end{equation}
Coincidentally, we shall see that the characteristic polynomial
$\chi({\mathcal A}_\FL(G,\varepsilon;A),t)$ is equal to the {\em
(modular) flow polynomial} $\varphi(G,t)$, defined for $t=q$ as the
number of nowhere-zero flows of $(G,\varepsilon)$ with values in an
abelian group of order $q$. The polynomial $\varphi$ is independent
of the chosen orientation $\varepsilon$ and the abelian group
structure; see Rota~\cite{Rota1} and Tutte \cite{Tutte1}. For a
complete information about modular and integral flows, we refer to
the book of Zhang \cite{Zhang1}.

Recall that a {\em cut} of $G$ is a nonempty edge subset of the form
$[S,S^c]$, where $S\subseteq V$ is a nonempty proper subset,
$S^c:=V-S$ is the complement of $S$, and $[S,S^c]$ is the set of all
edges between the vertices of $S$ and the vertices of $S^c$. Let
$U=[S,S^c]$ be a cut. A {\em direction} of $U$ is an {\em
orientation} $\varepsilon_U$ on the induced subgraph $(V,U)$ such
that the arrows of the edges in $U$ are either all from $S$ to $S^c$
or all from $S^c$ to $S$; $U$ together with a direction
$\varepsilon_U$ is called a {\em directed cut}, denoted
$(U,\varepsilon_U)$. If $(U,\varepsilon)$ is a directed cut, we call
$(U,\varepsilon)$ a directed cut of both $(G,\varepsilon)$ and
$\varepsilon$, and say that the cut $U$ is {\em directed} in
$(G,\varepsilon)$. Let ${\mathcal O}(G)$ denote the set of all
orientations on $G$. We denote by ${\mathcal O}_\textsc{tc}(G)$ the
set of all orientations without directed cut (also known as {\em
totally cyclic orientations}, as they are the orientations in which
every edge belongs to a directed circuit).

Let $\varphi_{\mathbbm z}(G,q)$ denote the number of nowhere-zero
integer-valued flows $f$ of $(G,\varepsilon)$ such that $0<|f(e)|<q$
for all $e\in E$. As pointed out by Beck and Zaslavsky
\cite{Beck-Zaslavsky1}, the function $\varphi_{\mathbbm z}(G,q)$ was
never mentioned to be a polynomial until Kochol \cite{Kochol1}. If
$\varepsilon$ is totally cyclic, we introduce the counting functions
\begin{equation}
\varphi_{\varepsilon}(G,q):=\#\{f\in F(G,\varepsilon;{\Bbb Z})\:|\:
0< f(e)< q, e\in E\},
\end{equation}
\begin{equation}
\bar\varphi_{\varepsilon}(G,q):=\#\{f\in F(G,\varepsilon;{\Bbb
Z})\:|\: 0\leq f(e)\leq q, e\in E\},
\end{equation}
and the relatively open 0-1 polytope
\begin{equation}
\Delta^+_{\textsc{fl}}(G,\varepsilon):=\{f\in F(G,\varepsilon;{\Bbb
R})\:|\: 0<f(e)<1, e\in E\}.
\end{equation}
The closure $\bar\Delta^+_{\textsc{fl}}(G,\varepsilon)$ is a $0$-$1$
polytope (whose vertices are 0-1 vectors), and is the convex hull of
all $0$-$1$ flows of $(G,\varepsilon)$ in ${\Bbb R}^E$ (flows whose
values are either 0 or 1), called the {\em flow polytope} of
$(G,\varepsilon)$. We introduce the following counting function
\begin{align} \bar \varphi_{\mathbbm z}(G,q): &=
\#\{(\rho,f)\:|\:\rho\in {\mathcal O}_\textsc{tc}(G), f\in
F(G,\rho;{\Bbb Z}), 0\leq f(e)\leq q, e\in E\}.
\end{align}
We shall see that $\varphi_{\mathbbm z}(G,q)$,
$\varphi_{\varepsilon}(G,q)$ are polynomial functions of positive
integers $q$, and $\bar\varphi_{\mathbbm z}(G,q)$,
$\bar\varphi_{\varepsilon}(G,q)$ are polynomial functions of
nonnegative integers $q$, and that $\varphi_{\mathbbm z}(G,q)$ is
independent of the chosen orientation $\varepsilon$. The
corresponding polynomial $\varphi_{\mathbbm z}(G,t)$
($\bar\varphi_{\mathbbm z}(G,t)$) is called the  {\em (dual)
integral flow polynomial} of $G$, and $\varphi_{\varepsilon}(G,t)$
($\bar\varphi_{\varepsilon}(G,t)$) the {\em local (dual) flow
polynomial with respect to the orientation $\varepsilon$}. The names
and notations are so selected in order to easily recognize these
polynomials.

We first reproduce a result due to Kochol \cite{Kochol1} about
Equation (\ref{FOP}), and due to Beck and Zaslavsky
\cite{Beck-Zaslavsky1} about the combinatorial interpretation of the
values of $\varphi_{\mathbbm z}(G,t)$ at nonpositive integers.

\begin{thm}[Kochol \cite{Kochol1}, Beck and Zaslavsky \cite{Beck-Zaslavsky1}]\label{Integral-Flow-Thm}
Let $G=(V,E)$ be a finite bridgeless graph with possible loops and
multiple edges.
\begin{enumerate}
\item[(a)] If the orientation $\varepsilon$ is totally cyclic,
then $\Delta^+_\FL(G,\varepsilon)$ is a relatively open $0$-$1$
polytope in $\Bbb R^E$ of dimension $n(G)$;
$\varphi_{\varepsilon}(G,t)$ and $\bar{\varphi}_{\varepsilon}(G,t)$
are Ehrhart polynomials of $\Delta^+_\FL(G,\varepsilon)$ and
$\bar\Delta^+_\FL(G,\varepsilon)$ respectively, and satisfy the {\em
Reciprocity Law:}
\begin{equation}\label{FPRL}
\varphi_{\varepsilon}(G,-t)=(-1)^{n(G)}\bar{\varphi}_{\varepsilon}(G,t),
\end{equation}
where $n(G)=|E|-r(G)$ and $r(G)$ is the number of edges of a maximal
spanning forest of $G$. Moreover,
\[
\varphi_{\varepsilon}(G,0)=(-1)^{n(G)},\sp
\bar{\varphi}_{\varepsilon}(G,0)=1.
\]

\item[(b)] The integral flow polynomials $\varphi_{\mathbbm z}(G,t)$ and
$\bar{\varphi}_{\mathbbm z}(G,t)$ can be written as
\begin{eqnarray}
\varphi_{\mathbbm z}(G,t) &=& \sum_{\rho\in {\mathcal
O}_\textsc{tc}(G)}
\varphi_{\rho}(G,t),\label{FOP}\\
\bar{\varphi}_{\mathbbm z}(G,t) &=& \sum_{\rho\in {\mathcal
O}_\textsc{tc}(G)}\bar{\varphi}_{\rho}(G,t), \label{FCP}
\end{eqnarray}
and satisfy the {\em Reciprocity Law:}
\begin{equation}\label{phiRL}
\varphi_{\mathbbm z}(G,-t)=(-1)^{n(G)}\bar{\varphi}_{\mathbbm
z}(G,t).
\end{equation}
In particular, $|\varphi_{\mathbbm z}(G,0)|$ counts the number of
totally cyclic orientations on $G$.
\end{enumerate}
\end{thm}

There are analogous results on the modular flow polynomial
$\varphi(G,t)$. To do this we need to introduce an equivalence
relation on the set ${\mathcal O}(G)$ of orientations on $G$. Two
orientations $\varepsilon_1,\varepsilon_2$ on $G$ are said to be
{\em Eulerian equivalent}, written $\varepsilon_1\sim\varepsilon_2$,
if the spanning subgraph induced by the edge subset $\{e\in
E\:|\:\varepsilon_1(v,e)\neq\varepsilon_2(v,e)\}$ is a directed
Eulerian subgraph with respect to the orientation either
$\varepsilon_1$ or $\varepsilon_2$. We shall see that $\sim$ is
indeed an equivalence relation on ${\mathcal O}(G)$. Moreover, if an
Eulerian equivalence class intersects ${\mathcal O}_\textsc{tc}(G)$,
the whole equivalence class is contained in ${\mathcal
O}_\textsc{tc}(G)$. So $\sim$ induces an equivalence relation on the
set ${\mathcal O}_\textsc{tc}(G)$ of totally cyclic orientations.
Let $[{\mathcal O}_\textsc{tc}(G)]$ denote a set of distinct
representatives, exact one representative from each equivalence
class of $\sim$ on ${\mathcal O}_\textsc{tc}(G)$. We introduce the
following counting function
\begin{align}
\bar{\varphi}(G,q): &= \#\{(\rho,f)\:|\: \rho\in [{\mathcal
O}_\textsc{tc}(G)],f\in F(G,\rho;{\Bbb Z}), 0\leq f(e)\leq q, e\in
E\}.
\end{align}
We next produce the following Theorem~\ref{Modular-Flow-Thm}, which
answers a question by Beck and Zaslavsky \cite{Beck-Zaslavsky1}
about the combinatorial interpretation of the values of the modular
flow polynomial $\varphi(G,t)$ at zero and negative integers.

\begin{thm}\label{Modular-Flow-Thm}
Let $G=(V,E)$ be a finite bridgeless graph with possible loops and
multiple edges. Then $\varphi(G,q)$ ($\bar{\varphi}(G,q))$ is a
polynomial function of degree $n(G)$ of positive (nonnegative)
integers $q$, and satisfy the {\em Reciprocity Law:}
\begin{equation}\label{Modular-Flow-Reciprocity-Law}
\varphi(G,-t)=(-1)^{n(G)}\bar{\varphi}(G,t).
\end{equation}
Moreover,
\begin{equation}\label{Modular-Flow-Sum}
\varphi(G,t)=\sum_{\rho\in [{\mathcal O}_\textsc{tc}(G)]}
\varphi_{\rho}(G,t),
\end{equation}
\begin{equation}\label{Closed-Modular-Flow-Sum}
\bar{\varphi}(G,t)=\sum_{\rho\in [{\mathcal O}_\textsc{tc}(G)]}
\bar{\varphi}_{\rho}(G,t).
\end{equation}
In particular, $|\varphi(G,-1)|$ counts the number of totally cyclic
orientations on $G$, and $|\varphi(G,0)|$ counts the number of
Eulerian equivalence classes of totally cyclic orientations.
\end{thm}

Equation (\ref{Modular-Flow-Sum}) is recently obtained by Kochol
\cite{Kochol1} with a formal proof in different form. The
combinatorial interpretation of $|\varphi(G,-1)|=T_G(0,2)$ is due to
Las Vergnas \cite{LasVergnas1}, see also Brylawski and Oxley
\cite{Brylawski-Oxley1}. At the moment of this revising, we noticed
a paper by Breuer and Sanyal \cite{Breuer-Sanyal1} on modular flow
reciprocity, which is quite different from our Reciprocity Law
(\ref{Modular-Flow-Reciprocity-Law}). The difference lies in that
the result of \cite{Breuer-Sanyal1} on $\varphi(G,-q)$ for a
positive integer $q$ involves the counting of flows modulo $q$, our
result only involves nonnegative integer flows bounded by $q$, and
that the bijection between the two counting sets is nontrivial; see
Section 5 for the detailed discussion. The proof of
Theorem~\ref{Modular-Flow-Thm} is rigorous and self-contained. The
following corollary is an immediate consequence of
Theorem~\ref{Modular-Flow-Thm}.

\begin{cor}
The value $T_G(0,1)$ of the Tutte polynomial $T_G(x,y)$ counts the
number of Eulerian equivalence classes of totally cyclic
orientations on $G$.
\end{cor}

\section{Characteristic polynomials of group arrangements}

Let $\Omega$ be a finitely generated abelian group. By a {\em flat}
of $\Omega$ we mean a coset of a subgroup of $\Omega$. For a
subgroup $\Gamma\subseteq\Omega$, we denote by $\Tor(\Gamma)$ the
torsion subgroup of $\Gamma$ and write
$|\Gamma|:=|\Tor(\Gamma)|\,t^{\rk(\Gamma)}$. By a {\em subgroup
arrangement} (or just {\em arrangement}) of $\Omega$ we mean a
finite collection of flats of $\Omega$. Associated with a subgroup
arrangement $\mathcal A$ is the semilattice $\mathscr L(\mathcal
A)$, whose members are nonempty sets obtained from all possible
intersections of flats in $\mathcal A$. The {\em characteristic
polynomial} of $\mathcal A$ is defined as
\begin{equation}\label{Defn-Char}
\chi(\mathcal A,t) = \sum_{X\in\mathscr L(\mathcal A)}
\frac{|\Tor(\Omega)|}{|\Tor(\Omega/\langle
X\rangle)|}\,\mu(X,\Omega) \,t^{\rk\langle X\rangle},
\end{equation}
where $\mu$ is the M\"{o}bius function of the poset $\mathscr
L({\mathcal A})$, whose partial order is the set inclusion, $\langle
X\rangle:=\{x-y\;|\;x,y\in X\}$.

Let ${\mathscr B}(\Omega)$ be the Boolean algebra generated by
cosets of all subgroups of $\Omega$, i.e., every member of $\mathscr
B(\Omega)$ is obtained from cosets of subgroups of $\Omega$ by
taking unions, intersections, and complements finitely many times. A
{\em valuation} on $\Omega$ with values in an abelian group $A$ is a
map $\nu:\mathscr B(\Omega)\rightarrow A$ such that
\begin{gather}
\nu(\emptyset) =0, \nonumber\\
\nu(X\cup Y) =\nu(X)+\nu(Y)-\nu(X\cap Y) \nonumber
\end{gather}
for $X,Y\in{\mathscr B}(\Omega)$. A valuation $\nu$ is said to be
{\em translation invariant} if
\[
\nu(S+x)=\nu(S)
\]
for $S\in{\mathscr B}(\Omega)$ and any $x\in\Omega$; and $\nu$ is
said to satisfy {\em multiplicativity} if
\[
\nu(A+B)=\nu(A)\,\nu(B)
\]
for subgroups $A,B\subseteq\Omega$ such that $A+B$ is a direct sum
of $A$ and $B$, and the subgroup $A+B$ is a direct summand of
$\Omega$.

\begin{thm}[Chen \cite{OLS-I}]\label{Subgroup-Valuation}
For any finitely generated abelian group $\Omega$, there exists a
unique translation invariant valuation $\lambda:\mathscr
B(\Omega)\rightarrow\Bbb Q[t]$ such that the multiplicativity is
satisfied and
\[
\lambda(\Omega)=|\Tor(\Omega)|\,t^{\rk(\Omega)}=|\Omega|.
\]
In particular, $\lambda(\Gamma)=\frac{|\Omega|}{|\Omega/\Gamma|}$
for any subgroup $\Gamma\subseteq\Omega$, and for any subgroup
arrangement $\mathcal A$ of $\Omega$,
\[
\lambda\bigg(\Omega-\bigcup_{X\in\mathcal A}X\bigg) =\chi(\mathcal
A,t).
\]
\end{thm}

The analogue of Theorem~\ref{Subgroup-Valuation} for vector spaces
was obtained by Ehrenborg and Readdy \cite{Ehrenborg-Readdy1}. Let
$V$ be a vector space over an infinite field. Let $\mathscr L(V)$ be
the lattice of all affine subspaces of $V$. We denote by $\mathscr
B(V)$ the Boolean algebra generated by $\mathscr L(V)$. A {\em
subspace arrangement} of $V$ is a finite collection $\mathcal A$ of
affine subspaces of $V$. The {\em torsion} of any subspace is just
the zero space. Then {\em characteristic polynomial} $\chi(\mathcal
A,t)$ of a subspace arrangement $\mathcal A$ can be defined by the
same formula \eqref{Defn-Char} for subgroup arrangement.

\begin{thm}[Ehrenborg and Readdy \cite{Ehrenborg-Readdy1}]\label{Subspace-Valuation}
For any finite-dimensional vector space $V$ over an infinite field
$\Bbb K$, there exists a unique translation invariant valuation
$\lambda:\mathscr B(V)\rightarrow\Bbb Z[t]$ such that
$\lambda(W)=t^{\dim W}$ for subspaces $W\subseteq V$. Moreover, for
a subspace arrangement $\mathcal A$ of $V$,
\[
\lambda\bigg(V-\bigcup_{X\in\mathcal A}X\bigg) =\chi(\mathcal A,t).
\]
\end{thm}

One may combine Theorems~\ref{Subgroup-Valuation} and
\ref{Subspace-Valuation} by considering arrangements of affine
submodules. Let $M$ be a finitely generated left $R$-module over a
commutative ring $R$; we restrict $R$ to the cases of ${\Bbb R}$,
${\Bbb Z}$, and ${\Bbb Z}/q{\Bbb Z}$. By a {\em flat} of $M$ we mean
a subset of the form $a+N=\{a+x\:|\: x\in N\}$, where $N$ is a
submodule of $M$. Let $\mathscr L(M)$ be the lattice of all flats of
$M$, and $\mathscr B(M)$ the Boolean algebra generated by $\mathscr
L(M)$. For each subset $S\subseteq M$, we denote by $1_S$ the
characteristic function of $S$.

Let $\mathcal A$ be a finite collection of flats in $M$, called a
{\em submodule arrangement} of $M$. Let $\mathscr L({\mathcal A})$
be the poset whose members are nonempty sets obtained by taking all
possible intersections of flats in $\mathcal A$, and whose partial
order $\leq$ is the set inclusion. For each $X\in\mathscr
L({\mathcal A})$, we define
\[
X^\circ:=X-\bigcup_{Y\in\mathscr L({\mathcal A}),\, Y<X}Y.
\]
Clearly, $\{X^\circ\:|\: X\in\mathscr L({\mathcal A})\}$ is a family
of disjoint subsets of $M$. Then for each $X\in \mathscr L({\mathcal
A})$,
\[
1_X=\sum_{Y\in \mathscr L({\mathcal A}),\,Y\leq X} 1_{Y^\circ}.
\]
By the M\"{o}bius inversion, for each $X\in \mathscr L({\mathcal
A})$,
\[
1_{X^\circ}=\sum_{Y\in\mathscr L({\mathcal A}),\,Y\leq X}
\mu(Y,X)1_Y.
\]
In particular, $M^\circ=M-\bigcup{\mathcal
A}=M-\bigcup_{X\in\mathcal A}X$ and
\begin{equation}\label{SF}
1_{M-\bigcup{\mathcal A}}=\sum_{Y\in\mathscr L({\mathcal A})}
\mu(Y,M)1_Y.
\end{equation}
Thus for any valuation $\nu$ on $\mathcal{B}(M)$, we have the
Inclusion-Exclusion Formula:
\begin{equation}\label{Incl-Excl-Valuation}
\nu\big(M-\mbox{$\bigcup$}{\mathcal A}\big)=\sum_{X\in \mathscr
L({\mathcal A})}\mu(X,M)\nu(X).
\end{equation}
This is a prototype of many existing formulas when $\nu$ is taken to
be various valuations; see \cite{Char-poly, Ehrenborg-Readdy1,
Zaslavsky2}.

Let $M$ be the Euclidean $n$-space ${\Bbb R}^n$. One has half-spaces
$\{x\in{\Bbb R}^n\:|\: f(x)\leq c\}$ (with linear functionals
$f:{\Bbb R}^n\rightarrow{\Bbb R}$ and constant real numbers $c$),
convex polyhedra (intersections of half-spaces), and the Boolean
algebra $\mathscr B({\mathcal P}^n)$ (generated by half-spaces by
taking intersections, unions, and relatively complements finitely
many times). There are two valuations $\chi$ and $\bar\chi$ on
$\mathscr B({\mathcal P}^n)$, both are referred to the {\em Euler
characteristic} (see \cite{Euler-char,Schanuel1,Zaslavsky2}, for
example), such that for any relatively open convex polyhedron $P$,
\[
\chi(P)=(-1)^{\dim P}, \sp
\bar\chi(P)=\lim_{r\rightarrow\infty}\chi(P\cap[-r,r]^n).
\]
By Groemer's extension theorem \cite{Groemer1}, $\chi$ and
$\bar\chi$ can be extended to be linear functionals on the
functional space spanned by characteristic functions of convex
polyhedra. Now let $\mathcal A$ be a hyperplane arrangement of
${\Bbb R}^n$. Evaluating $\chi$ and $\bar\chi$ on both sides of
(\ref{Incl-Excl-Valuation}), one obtains Zaslavsky's first and
second counting formulas (see
\cite{Ehrenborg-Readdy1,Zaslavsky1,Zaslavsky2}):
\begin{equation}\label{Zaslavsky-First-Formula}
|\chi({\mathcal A},-1)|=\mbox{number of regions of}\;{\Bbb
R}^n-\mbox{$\bigcup$} \mathcal{A},
\end{equation}
\begin{equation}\label{Zaslavsky-Second-Formula}
|\chi({\mathcal A},1)| = \mbox{number of relatively bounded regions
of}\; {\Bbb R}^n-\mbox{$\bigcup$} \mathcal{A}.
\end{equation}

\section{Modular flow polynomials}

Let $(H_i,\varepsilon_i)$ be subdigraphs of the graph $G=(V,E)$,
$i=1,2$. The {\em coupling} of $\varepsilon_1$ and $\varepsilon_2$
is a function
$[\varepsilon_1,\varepsilon_2]:E\rightarrow\{-1,0,1\}$, defined for
each edge $e\in E$ (at its one end-vertex $v$) by
\begin{equation}
[\varepsilon_1,\varepsilon_2](e)=\left\{\begin{array}{rl} 1 &
\mbox{if $e\in E(H_1)\cap E(H_2)$, $\varepsilon_1(v,e)=\varepsilon_2(v,e)$,} \\
-1 & \mbox{if $e\in E(H_1)\cap E(H_2)$, $\varepsilon_1(v,e)\neq \varepsilon_2(v,e)$,}\\
0 & \mbox{otherwise}.
\end{array}\right.
\end{equation}
The following proposition is straightforward.

\begin{prop}[Berge \cite{Berge}]\label{Characterization-of-Eulerian}
\begin{enumerate}
\item[(a)] A function $f:E(G)\rightarrow A$ is a flow of $(G,\varepsilon)$ if
and only if for any directed cut $(U,\varepsilon_U)$,
\[
\sum_{e\in U}[\varepsilon,\varepsilon_U](e)f(e)=0.
\]
\item[(b)] In particular, the digraph $(G,\varepsilon)$ is directed Eulerian if
and only if for any directed cut $(U,\varepsilon_U)$,
\[
\sum_{e\in U}[\varepsilon,\varepsilon_U](e)=0.
\]

\item[(c)] The graph $G$ is Eulerian if and only if every cut $U$ contains even
number of edges.
\end{enumerate}
\end{prop}

Let $F_{\rm nz}(G,\varepsilon;A)$ denote the set of all nowhere-zero
flows with values in $A$, i.e.,
\[
F_{\rm nz}(G,\varepsilon;A): =\{f\in F(G,\varepsilon;A)\:|\:
f(e)\neq 0,\:e\in E\}.
\]
If $|A|=q$ is finite, it is well-known (see \cite{Zhang1}) that the
counting function
\begin{equation}
\varphi(G,q):=|F_{\rm nz}(G,\varepsilon;A)|
\end{equation}
is a polynomial function of $q$, depending only on the order $|A|$,
but not on the chosen orientation $\varepsilon$ and the group
structure of $A$. The polynomial $\varphi(G,t)$ is called the {\em
modular flow polynomial} of $G$.

For two orientations $\rho,\sigma\in{\mathcal O}(G)$, there is an
involution $P_{\rho,\sigma}:A^E\rightarrow A^E$, defined by
\begin{equation}\label{TF}
\left(P_{\rho,\sigma} f\right)(e)=\left\{\begin{array}{rl}
f(e) & \mbox{\rm if $\rho(v,e)=\sigma(v,e)$,}\\
-f(e) & \mbox{\rm if $\rho(v,e)\neq\sigma(v,e)$.}
\end{array}\right.
\end{equation}
In fact, $P_{\rho,\varepsilon}f=[\rho,\varepsilon]f$. Obviously,
$P_{\rho,\rho}$ is the identity map,
$P_{\rho,\sigma}P_{\sigma,\varepsilon}=P_{\rho,\varepsilon}$.

\begin{lemma}The involution
$P_{\rho,\varepsilon}$ is a group isomorphism. Moreover,
\begin{eqnarray*}
P_{\rho,\varepsilon}F(G,\varepsilon;A) &=& F(G,\rho;A),\\
P_{\rho,\varepsilon}F_{\rm nz}(G,\varepsilon;A) &=& F_{\rm
nz}(G,\rho;A).
\end{eqnarray*}
\end{lemma}
\begin{proof}
It is clear that $P_{\rho,\varepsilon}$ is invertible and
$P^{-1}_{\rho,\varepsilon}=P_{\rho,\varepsilon}$. Let $f\in A^E$.
The group isomorphism follows from the fact that at each vertex $v$,
\[
\sum_{e\in E}\rho(v,e)(P_{\rho,\varepsilon}f)(e) =\sum_{e\in
E}\rho(v,e)\rho(v,e)\varepsilon(v,e)f(e) =\sum_{e\in
E}\varepsilon(v,e)f(e).
\]
Since $P_{\rho,\varepsilon}f(e)\neq 0$ is equivalent to $f(e)\neq
0$, it follows that $P_{\rho,\varepsilon}F_{\rm
nz}(G,\varepsilon;A)=F_{\rm nz}(G,\rho;A)$.
\end{proof}

Let $T$ be a {\em maximal forest} of $G$ in the sense that every
component of $T$ is a spanning tree of a component of $G$. For each
edge $e$ of the complement $T^c:=E-E(T)$, let $C_e$ denote the
unique circuit contained in $T\cup e$, and let $\rho_e$ be a
direction of $C_e$ (i.e. $(C_e,\rho_e)$ is directed Eulerian) such
that $\rho_e(e)=\varepsilon(e)$. It is easy to see that
$[\varepsilon,\rho_e]$ is a flow of $(G,\varepsilon)$.

\begin{lemma}[Berge \cite{Berge}]\label{Rank-Flow-Space}
Let $T$ be a maximal spanning forest of $G$. Then each flow $f$ of
the digraph $(G,\varepsilon)$ can be expressed as a unique linear
combination
\begin{equation}\label{circuit-eq}
f = \sum_{e\in T^c} f(e)[\varepsilon,\rho_e].
\end{equation}
\end{lemma}

The system \eqref{Balance-Equation} for a flow $f$, whose equations
are indexed by vertices $v\in V$, is equivalent to the system
\eqref{circuit-eq}, which can be written as
\begin{equation}\label{circuit-eq-1}
f(x) = \sum_{e\in T^c} f(e)[\varepsilon,\rho_e](x), \sp x\in T,
\end{equation}
whose equations are indexed by edges $x\in T$. In other words, the
values $f(e)$ for $e\in T^c$ can be arbitrarily specified, and
$f(x)$ for $x\in T$ are determined by \eqref{circuit-eq-1}.

Let $n(G)$ denote the cycle rank of $G$. If $T$ is a maximal
spanning forest of $G$, then $n(G)$ is the number of edges of $T^c$.
Note that
\[
n(G)=|E|-|V|+c(G),
\]
where $c(G)$ is the number of connected components of $G$.
Lemma~\ref{Rank-Flow-Space} shows that the abelian group
$F(G,\varepsilon;A)$ is of rank $n(G)$. The {\em flow arrangement}
of the digraph $(G,\varepsilon)$ with the abelian group $A$ is the
subgroup arrangement ${\mathcal A}_{\FL}(G,\varepsilon;A)$ of
$F(G,\varepsilon;A)$, consisting of the subgroups
\begin{equation}
F_e\equiv F_e(G,\varepsilon;A):=\{f\in F(G,\varepsilon;A)\:|\:
f(e)=0\}, \sp e\in E.
\end{equation}
The semilattice $\mathscr{L}({\mathcal A}_{\FL}(G,\varepsilon;A))$
consists of the subgroups
\[
F_X\equiv F_X(G,\varepsilon;A) :=\{f\in F(G,\varepsilon;A)\:|\:
f(x)=0,\,x\in X\}, \sp X\subseteq E.
\]
Then $F\equiv F_\emptyset=F(G,\varepsilon;A)$. Using
\eqref{circuit-eq-1}, it is easy to see that the arrangement
${\mathcal A}_\FL(G,\varepsilon;A)$ is isomorphic to the arrangement
${\mathcal A}_\FL(G,T,\varepsilon;A)$ of $A^{E(T^c)}$, consisting of
the subgroups
\[
H_e:= \left\{\begin{array}{ll}
\{f\in A^{E(T^c)}\:|\: f(e)=0\} & \mbox{if $e\in T^c$,}\\
\{f\in A^{E(T^c)}\:|\: \sum_{x\in T^c}
[\varepsilon,\rho_x](e)f(x)=0\} & \mbox{if $e\in T$.}
\end{array}\right.
\]
The isomorphism is given by the restriction $f\mapsto f|_{E(T^c)}$,
sending $F_e$ to $H_e$.

\begin{lemma}\label{q^r}
The abelian group $F(G,\varepsilon;A)$ is isomorphic to the product
group $A^{n(G)}$. Moreover, for any subset $X\subseteq E$, the
subgroup $F_X(G,\varepsilon;A)$ is isomorphic to the product group
$A^{n\langle E-X\rangle}$. In particular, if $|A|=q$ is finite, then
\begin{equation}
|F(G,\varepsilon;A)|=q^{n(G)}, \sp |F_X(G,\varepsilon;A)|=q^{n\langle
E-X\rangle}.
\end{equation}
\end{lemma}
\begin{proof}
Trivial
\end{proof}

\begin{thm}\label{MFP}
Let $A$ be an abelian group such that either $|A|=q$ is finite, or $A=\Bbb Z$,
or $A$ is an infinite field. Then
\begin{equation}\label{Modular-Flow-Characteristic-Formula}
\varphi(G,q) =\lambda\bigg(F(G,\varepsilon;A)-\bigcup_{e\in
E}F_e\bigg)\bigg|_{t=q} =\chi({\mathcal A}_\FL(G,\varepsilon;A),q).
\end{equation}
Moreover, $|\varphi(G,-1)|$ counts the number of totally cyclic
orientations on $G$.
\end{thm}
\begin{proof}
Let $X,Y\subseteq E$. If $|A|\geq 2$ (including $|A|=\infty$), then
$F_X(G,\varepsilon;A)\subseteq F_Y(G,\varepsilon;A)$ is equivalent
to that circuits of $\langle E-X\rangle$ are contained in $\langle
E-Y\rangle$. Thus the map $F_X(G,\varepsilon;A)\mapsto
F_X(G,\varepsilon;\Bbb R)$ is an isomorphism from $\mathscr
L(\mathcal A_\FL(G,\varepsilon;A))$ to $\mathscr L(\mathcal
A_\FL(G,\varepsilon;\Bbb R))$. Consequently, we have the same
M\"{o}bius function
\[
\mu(F_X(G,\varepsilon;A),F_Y(G,\varepsilon;A)) =
\mu(F_X(G,\varepsilon;\Bbb R),F_Y(G,\varepsilon;\Bbb R)).
\]
If $A$ is infinite, applying the valuation $\lambda$ to both sides
of \eqref{SF},  we have
\[
\lambda\bigg(F-\bigcup_{e\in E}F_e\bigg) =\sum_{F_X\in\mathscr
L(\mathcal A_\FL(G,\varepsilon;A))} \mu(F_X,F)\,t^{n\langle
E-X\rangle}.
\]
If $A$ is finite and $|A|=q$, applying the counting measure $\#$ to
both sides of \eqref{SF}, we have
\[
\varphi(G,q) =\sum_{F_X\in\mathscr L(\mathcal
A_\FL(G,\varepsilon;A))} \mu(F_X,F)\,q^{n\langle E-X\rangle}.
\]
The identity \eqref{Modular-Flow-Characteristic-Formula} follows
immediately for positive integers $q$.

For $q=1$, we have $A=\{0\}$ and $F_{\rm
nz}(G,\varepsilon;A)=\emptyset$. Hence $\varphi(G,1)=0$. Since the
hyperplane arrangement $\mathcal A_\FL(G,\varepsilon;\Bbb R)$ is
central (all hyperplanes pass through the origin), there is no
relatively bounded region. Zaslavsky's second counting formula
\eqref{Zaslavsky-Second-Formula} confirms that $\chi(\mathcal
A_\FL(G,\varepsilon;\Bbb R),1)=0$.

Finally, by Zaslavsky's counting formula
\eqref{Zaslavsky-First-Formula}, $(-1)^{n(G)}\varphi(G,-1)$ counts
the number of regions of the complement $F-\bigcup_{e\in E} F_e$. By
Lemmas~\ref{PRE} and \ref{0-1-Polytope}, the regions of the
complement $F(G,\varepsilon;{\Bbb R})-\bigcup{\mathcal
A}_\FL(G,\varepsilon;{\Bbb R})$ correspond bijectively to the
totally cyclic orientations on $G$.
\end{proof}

\section{Integral flow polynomials}

In this section we apply the Ehrhart polynomial theory to study
integral flow polynomials. Let us recall some well-known facts about
lattice polytopes and Ehrhart polynomials. Let $P$ be a relatively
open lattice polytope of ${\Bbb R}^n$, i.e., $P$ is open in the flat
that $P$ spans, and the vertices of $P$ are lattice points of ${\Bbb
Z}^n$. The closure of $P$ is denoted by $\bar P$. A bounded {\em
lattice polyhedron} is a disjoint union of finitely many relatively
open lattice polytopes. Let $X$ be a bounded lattice polyhedron and
$q$ a positive integer. The {\em dilatation} of $X$ by $q$ is the
polyhedron $qX:=\{qx\:|\:x\in X\}$. Let
\[
L(X,q):= \#(qX\cap {\Bbb Z}^n).
\]
It is known that $L(X,q)$ is a polynomial function of degree $\dim
X$ in the positive integer variable $q$, called the {\em Ehrhart
polynomial} of $X$. Moreover, the leading coefficient of $L(X,t)$ is
the volume of $X$; the constant term $L(X,0)$ is the Euler
characteristic $\chi(X)$. In particular, if $X=P$ is a relatively
open lattice polytope, then $L(P,q)$ and $L(\bar{P},q)$ satisfy the
{\em Reciprocity Law:}
\[
L(P,-t)=(-1)^{\dim P}L\bigl(\bar{P},t\bigr);
\]
the constant term of $L(\bar P,t)$ is $1$, and the constant term of $L(P,t)$ is
$(-1)^{\dim P}$. All these and other related properties about Ehrhart
polynomials can be found in \cite{Lattice-points,Ehrhart-poly,Stanley2}.

Flows with values in ${\Bbb R}$ are called {\em real flows}; and
flows with values in ${\Bbb Z}$ are called {\em integer flows}. A
flow $f\in F(G,\varepsilon;{\Bbb R})$ is called a {\em $q$-flow} if
$|f(e)|<q$ for all $e\in E$. We define the set of all real $q$-flows
of $(G,\varepsilon)$ as
\[
F(G,\varepsilon;q): =\{f\in F(G,\varepsilon;\Bbb R): |f(e)|<q, e\in
E\}.
\]
We denote by $F_{\mathbbm z}(G,\varepsilon;q)$ the set of all
integer $q$-flows of $(G,\varepsilon)$, and by $F_{\rm nz\mathbbm
z}(G,\varepsilon;q)$ the set of all nowhere-zero integer $q$-flows,
i.e.,
\[
F_{\rm nz\mathbbm z}(G,\varepsilon;q):=\{f\in F_{\mathbbm
z}(G,\varepsilon;q)\:|\: f(e)\neq 0, e\in E\}.
\]
Clearly, $F_{\rm nz\mathbbm z}(G,\varepsilon;q)$ is the set of
lattice points of the dilatation $q\Delta_\FL(G,\varepsilon)$
(dilated by $q$) of the nonconvex polyhedron
\[
\Delta_\FL(G,\varepsilon): =\{f\in F(G,\varepsilon;{\Bbb R}):
0<|f(e)|<1,e\in E\}.
\]
It follows that the counting function
\begin{equation}
\varphi_{\mathbbm z}(G,q):=|F_{\rm nz\mathbbm
z}(G,\varepsilon;q)|=L(\Delta_\FL(G,\varepsilon),q)
\end{equation}
is an Ehrhart polynomial of degree $\dim\Delta_\FL(G,\varepsilon)$
in the positive integer variable $q$. In fact, we shall see that
$|F_{\rm nz\mathbbm z}(G,\varepsilon;q)|$ is independent of the
chosen orientation $\varepsilon$. We call $\varphi_{\mathbbm
z}(G,t)$ the {\em integral flow polynomial} of $G$.

\begin{lemma}\label{Involution-EP}
The involution $P_{\rho,\varepsilon}$ is a group isomorphism from
${\Bbb R}^E$ to itself. Moreover,
\begin{eqnarray*}
P_{\rho,\varepsilon} \Delta_{\FL}(G,\varepsilon) &=&
\Delta_{\FL}(G,\rho),\\
P_{\rho,\varepsilon} F_{\rm nz\mathbbm z}(G,\varepsilon;q) &=&
F_{\rm nz\mathbbm z}(G,\rho;q).
\end{eqnarray*}
\end{lemma}
\begin{proof}
Let $A={\Bbb R}$, $f\in{\Bbb R}^E$, $e\in E$. Note that
$0<|P_{\rho,\varepsilon}f(e)|<1$ is equivalent to $0<|f(e)|<1$.
Hence $P_{\rho,\varepsilon} \Delta_{\FL}(G,\varepsilon) =
\Delta_{\FL}(G,\rho)$. Similarly, $P_{\rho,\varepsilon}f(e)\in{\Bbb
Z}$ is equivalent to $f(e)\in{\Bbb Z}$; and
$0<|P_{\rho,\varepsilon}f(e)|<q$ is equivalent to $0<|f(e)|<q$. Thus
$P_{\rho,\varepsilon} F_{\rm nz\mathbbm z}(G,\varepsilon;q) = F_{\rm
nz\mathbbm z}(G,\rho;q)$.
\end{proof}

Let $A=\Bbb R$. The subgroup arrangement ${\mathcal
A}_\FL(G,\varepsilon;{\Bbb R})$ is a hyperplane arrangement of
$F(G,\varepsilon;{\Bbb R})$. The complement of ${\mathcal
A}_\FL(G,\varepsilon;{\Bbb R})$ is the set
\[
F(G,\varepsilon;{\Bbb R}) - \mbox{$\bigcup$}{\mathcal
A}_\FL(G,\varepsilon;{\Bbb R}) =\{f\in F(G,\varepsilon;{\Bbb
R})\:|\: f(e)\neq 0, e\in E\}.
\]
For each edge $e\in E$ with end-vertices $u,v$, the nonzero
condition $f(e)\neq 0$ can be split into two inequalities:
\[
f(e)>0 \sp\mbox{and} \sp f(e)<0;
\]
the former can be interpreted as an orientation of $e$ agreeing with
$\varepsilon(u,e)$, and the latter is interpreted as an orientation
of $e$ opposite to $\varepsilon(u,e)$.

For each orientation $\rho\in\mathcal O(G)$, we introduce the open
convex cone
\[
C^\rho(G,\varepsilon):=\{f\in F(G,\varepsilon;\Bbb R)
:[\rho,\varepsilon](e)f(e)>0, e\in E\}.
\]
The complement $F(G,\varepsilon;{\Bbb R})-\bigcup {\mathcal
A}_{\FL}(G,\varepsilon;{\Bbb R})$ is a disjoint union of these open
convex cones, some of them may be empty. By Lemma~\ref{PRE} below,
the cone $C^\rho(G,\varepsilon)$ is isomorphic to the open convex
cone
\[
C^+(G,\rho):=\{f\in F(G,\rho;{\Bbb R})\:|\: f(e)>0, e\in E\}.
\]
We introduce the relatively open polytopes
\[
\Delta^+_{\FL}(G,\rho):=\{f\in F(G,\rho;{\Bbb R})\:|\: 0<f(e)<1,
e\in E\},
\]
\[
\Delta^\rho_{\FL}(G,\varepsilon): =\{f\in F(G,\varepsilon;{\Bbb
R})\:|\: 0<[\rho,\varepsilon](e)f(e)<1, e\in E\}.
\]
If  $\Delta^+_\FL(G,\rho)\neq\emptyset$ (equivalent to that $\rho$
is totally cyclic), then the closure of $\Delta^+_\FL(G,\rho)$ is
the polytope
\begin{equation}
\bar\Delta^+_\FL(G,\rho):=\{f\in F(G,\rho;{\Bbb R})\:|\: 0\leq
f(e)\leq 1, e\in E\}.
\end{equation}
Whether the orientation $\rho$ is totally cyclic or not, the set
$\bar\Delta^+_\FL(G,\rho)$ is always a polytope, and is called a
{\em flow polytope of $G$ with respect to $\rho$}.

\begin{lemma}\label{PRE}
$P_{\rho,\varepsilon} \Delta^\rho_\FL(G,\varepsilon) =
\Delta^+_\FL(G,\rho)$; and disjoint decomposition
\begin{eqnarray}
\Delta_\FL(G,\varepsilon) &=& \bigsqcup_{\rho\in{\mathcal O}(G)}
\Delta^\rho_\FL(G,\varepsilon). \label{Flow-Decomposition}
\end{eqnarray}
\end{lemma}
\begin{proof}
Since $P_{\rho,\varepsilon}f=[\rho,\varepsilon]f$ for $f\in{\Bbb
R}^E$, then the first identity is trivial by definition of
$\Delta^\rho_\FL(G,\varepsilon)$, $\Delta^+_\FL(G,\rho)$, and
Lemma~\ref{Involution-EP}.

Let $f\in\Delta_\FL(G,\varepsilon)$. We define an orientation $\rho$
on $G$ as follows: for each edge $e$ at its one end-vertex $v$, set
$\rho(v,e)=\varepsilon(v,e)$ if $f(e)>0$ and
$\rho(v,e)=-\varepsilon(v,e)$ if $f(e)<0$. Then $f\in
\Delta^\rho_\FL(G,\varepsilon)$. Conversely, each
$\Delta^\rho_\FL(G,\varepsilon)$ is obviously contained in
$\Delta_\FL(G,\varepsilon)$. The union is clearly disjoint.
\end{proof}

Notice that the open convex cone $C^+(G,\varepsilon)$ may be empty
for the given orientation $\varepsilon$. If $(G,\varepsilon)$
contains a directed cut, then it is impossible to have positive real
flows by Proposition~\ref{Characterization-of-Eulerian}(a), thus
$C^+(G,\varepsilon)=\emptyset$. To have
$C^+(G,\varepsilon)\neq\emptyset$, the orientation $\varepsilon$
must be totally cyclic.

\begin{lemma}\label{0-1-Polytope}
{\rm (a)} $\Delta^+_{\FL}(G,\varepsilon)\neq\emptyset$ if and only
if $\varepsilon\in{\mathcal O}_\textsc{tc}(G)$.

{\rm (b)} If $\varepsilon\in{\mathcal O}_\textsc{tc}(G)$, then
$\bar\Delta^+_{\FL}(G,\varepsilon)$ is a $0$-$1$ polytope, i.e., all
its vertices are flows of $(G,\varepsilon)$ with values in
$\{0,1\}$.

{\rm (c)} If $\varepsilon\in{\mathcal O}_\textsc{tc}(G)$, then every
flow of $(G,\varepsilon)$ with values in $\{0,1\}$ is a vertex of
$\bar\Delta^+_{\FL}(G,\varepsilon)$.
\end{lemma}
\begin{proof} (a) Trivial.

(b) Let $f$  be a vertex of $\bar\Delta^+_{\FL}(G,\varepsilon)$. It
is enough to show that $f$ is integer-valued. By Linear Programming
the vertex $(x_e=f(e):e\in E)$ is a unique solution of a linear
system of the form
\[
\begin{array}{rlll} \sum_{e\in E}{\bm m}_{v,e} x_e &=& 0, & v\in V,\\
x_e &=& a_e, & e\in E_f,
\end{array}
\]
where $a_e=0$ or $1$, $E_f$ is an edge subset, $|E_f|=|E|-n(G)$. The
system is equivalent to the linear system
\[
\sum_{e\in E'_f}{\bm m}_{v,e}x_e=b_v,\sp v\in V,
\]
where $E'_f:=E-E_f$, $b_v\in{\Bbb Z}$, and the rank of the matrix
$[{\bm m}_{v,e}]_{V\times E'_f}$ is $n(G)$. Since the incidence
matrix ${\bm M}=[{\bm m}_{v,e}]_{V\times E}$ is totally unimodular
(see \cite{Bondy-Murty1}, p.35), the submatrix $[{\bm
m}_{v,e}]_{V\times E'_f}$ is row equivalent to the matrix
$\big[{I\atop 0}\big]$ over $\Bbb Z$, where $I$ is the identity
matrix when $V$ is linearly labeled. It then follows that the
solution $(x_e=f(e):e\in E)$ is an integer vector.

(c) Notice that a flow of $(G,\varepsilon)$ with values in $\{0,1\}$
is just the characteristic function of the edge set of a directed
Eulerian subgraph of $(G,\varepsilon)$. Let $f$ be a flow with
values in $\{0,1\}$. Suppose $f$ is not a vertex of
$\bar\Delta^+_{\FL}(G,\varepsilon)$. Then there are distinct flows
$f_i$ of $(G,\varepsilon)$ with values in $\{0,1\}$ such that
$f=\sum_{i=1}^ka_if_i$, where $a_i>0$, $\sum_{i=1}^ka_i=1$, $k\geq
2$. Let $e$ be an edge such that $f(e)=1$. Then $f_i(e)=1$ for all
$i$; otherwise, say, $f_1(e)=0$, then
$f(e)=\sum_{i=1}^ka_if_i(e)\leq \sum_{i=2}^ka_i<1$, which
contradicts $f(e)=1$. Thus $f_i=f$ for all $i$; this is
contradictory to the distinctness of $f_i$.
\end{proof}

Recall that $\varphi_{\varepsilon}(G,q)$ for a positive integer $q$
is the number of integer flows of $(G,\varepsilon)$ with values in
$\{1,2,\ldots,q-1\}$. In other words, $\varphi_{\varepsilon}(G,q)$
is the number of integer flows of $(G,\varepsilon)$ with values in
the open interval $(0,q)$. Clearly, $\varphi_{\varepsilon}(G,q)$
counts the number of lattice points of the dilatation
$q\Delta^+_\FL(G,\varepsilon)$, i.e.,
\begin{equation}\label{Phi-Z-Epsilon}
\varphi_{\varepsilon}(G,q)= L(\Delta^+_\FL(G,\varepsilon),q).
\end{equation}
We call $\varphi_{\varepsilon}(G,q)$ the {\em local flow polynomial
of $G$ with respect to $\varepsilon$}. Analogously, let
$\bar{\varphi}_{\varepsilon}(G,q)$ denote the number of integer
flows of $(G,\varepsilon)$ with values in $\{0,1,\ldots,q\}$. In
other words, $\bar{\varphi}_{\varepsilon}(G,q)$ is the number of
integer flows of $(G,\varepsilon)$ with values in the closed
interval $[0,q]$. Then $\bar{\varphi}_{\varepsilon}(G,q)$ counts the
number of lattice points of $q\bar\Delta^+_\FL(G,\varepsilon)$,
i.e.,
\begin{equation}\label{Bar-Phi-Z-Epsilon}
\bar{\varphi}_{\varepsilon}(G,q)=
L(\bar\Delta^+_\FL(G,\varepsilon),q).
\end{equation}
We call $\bar{\varphi}_{\varepsilon}(G,q)$ the {\em local dual flow
polynomial of $G$ with respect to $\varepsilon$}. Now we denote by
$\bar{\varphi}_{\mathbbm z}(G,q)$ the number of pairs $(\rho,f)$,
where $\rho$ is a totally cyclic orientation on $G$ and $f$ is an
integer flow of $(G,\rho)$ with values in $\{0,1,\ldots,q\}$. We
call $\bar{\varphi}_{\mathbbm z}(G,q)$ the {\em dual integral flow
polynomial of $G$}.\vspace{3ex}

{\sc Proof of Theorem~\ref{Integral-Flow-Thm}.} \vspace{2ex}

(a) By Lemma~\ref{0-1-Polytope}, the closure
$\bar\Delta^+_\FL(G,\varepsilon)$ of the open polytope
$\Delta^+_\FL(G,\varepsilon)$ is the convex hull of the lattice
points $f\in{\Bbb Z}^E$ such that $f(e)\in\{0,1\}$ for all $e\in E$
and satisfying (\ref{Balance-Equation}). The Reciprocity Law and the
interpretation of the constant term follow from the Reciprocity Law
and the properties of Ehrhart polynomials.

(b) Note that $F_{\rm nz\mathbbm{
z}}(G,\varepsilon;q)=q\Delta_\FL(G,\varepsilon)$. By
Lemma~\ref{PRE}, we have a disjoint union
\begin{equation}\label{DFL-disjoint}
q\Delta_\FL(G,\varepsilon) = \bigsqcup_{\rho\in {\mathcal
O}_\textsc{tc}(G)} q\Delta^\rho_\FL(G,\varepsilon),
\end{equation}
where each lattice open polytope $\Delta^\rho_\FL(G,\varepsilon)$ is
isomorphic to the 0-1 open polytope $\Delta^+_\FL(G,\rho)$ by the
unimodular transformation $P_{\rho,\varepsilon}$. Then (\ref{FOP})
follows immediately from (\ref{DFL-disjoint}); (\ref{FCP}) follows
from definition of $\bar{\varphi}_{\mathbbm z}(G,q)$. The
Reciprocity Law (\ref{phiRL}) follows from (\ref{FPRL})-(\ref{FCP}).
The interpretation of the constant term $\varphi_{\mathbbm z}(G,0)$
follows from (\ref{FOP}) and
$\varphi_{\varepsilon}(G,0)=(-1)^{n(G)}$. \hfill{$\Box$}

\section{Interpretation of modular flow polynomial}

This section is devoted to interpreting the values of the modular
flow polynomial in a way similar to how the modular tension
polynomial was interpreted in \cite{OLS-I}. For the graph $G=(V,E)$
and a positive integer $q$, there is a modulo $q$ map
\[
\Mod_q:{\Bbb R}^E\rightarrow({\Bbb R}/q{\Bbb Z})^E, \sp
(\Mod_qf)(x)=f(x)\;(\mod q), \sp f\in{\Bbb R}^E.
\]
Then $\Mod_q({\Bbb Z}^E)=({\Bbb Z}/q{\Bbb Z})^E$ is a subgroup of
the toric group $({\Bbb R}/q{\Bbb Z})^E$, and
$\Mod_q(F(G,\varepsilon;\Bbb Z))$ is a subgroup of $\Mod_q(\Bbb
Z^E)$.

Given orientations $\rho,\sigma\in\mathcal O(G)$; there is an
involution $Q_{\rho,\sigma}:[0,q]^E\rightarrow[0,q]^E$ defined by
\[
(Q_{\rho,\sigma}g)(e)=\left\{\begin{array}{rl}
g(e)   & \mbox{if $\rho(v,e)=\sigma(v,e)$,}\\
q-g(e) & \mbox{if $\rho(v,e)\neq\sigma(v,e)$,}
\end{array}\right.
\]
where $g\in[0,q]^E$ and $v$ is an end-vertex of the edge $e$.
Clearly, $Q_{\rho,\sigma}$ is a bijection from $[0,q]^E$ to
$[0,q]^E$, and is also a bijection from $(0,q)^E$ to $(0,q)^E$,
where $(0,q)=\{x\in{\Bbb R}\:|\:0<x<q\}$. Moreover,
$Q_{\rho,\sigma}Q_{\sigma,\varepsilon}=Q_{\rho,\varepsilon}$.

Recall that two orientations $\varepsilon_1,\varepsilon_2\in\mathcal
O(G)$ are said to be {\em Eulerian equivalent}, written
$\varepsilon_1\sim\varepsilon_2$, if the induced spanning subdigraph
by the edge subset
\[
E(\varepsilon_1\neq\varepsilon_2):= \{e\in E\:|\:
\varepsilon_1(v,e)\neq \varepsilon_2(v,e),\mbox{$v$ is an end-vertex
of $e$}\}
\]
is a directed Eulerian subgraph with the orientation either
$\varepsilon_1$ or $\varepsilon_2$.

\begin{lemma}\label{QRE}
\begin{enumerate}
\item[(a)] The relation $\sim$ is an equivalence relation on ${\mathcal O}(G)$.

\item[(b)] Let $\rho,\sigma\in\mathcal O(G)$ be Eulerian equivalent. If
$\rho$ is totally cyclic, so is $\sigma$.

\item[(c)] Let $\rho,\sigma\in\mathcal O(G)$ and
$\rho\sim\sigma$. Then $Q_{\rho,\sigma}:
q\bar\Delta^+_{\FL}(G,\sigma) \rightarrow q\bar
\Delta^+_{\FL}(G,\rho)$ is a bijection, sending lattice points to
lattice points. In particular,
\[
Q_{\rho,\sigma}(q\Delta^+_{\FL}(G,\sigma)) =
q\Delta^+_{\FL}(G,\rho),
\]
\[
 \varphi_{\rho}(G,q) =\varphi_{\sigma}(G,q),
\]
\[
\bar\varphi_{\rho}(G,\varepsilon;q) =\bar\varphi_{\sigma}(G,q).
\]
\end{enumerate}
\end{lemma}
\begin{proof}
(a) The reflexivity and the symmetric property are obvious.
Transitivity is a straightforward computation. Note that a digraph
$(H,\rho)$ is Eulerian if and only if for all $v\in V(H)$,
\[
\sum_{e\in E(H)}\rho(v,e)=0.
\]
Let $\varepsilon_i\in{\mathcal O}(G)$ ($i=1,2,3$) be such that
$\varepsilon_1\sim\varepsilon_2$ and
$\varepsilon_2\sim\varepsilon_3$. Then
\[
\sum_{e\in
E(\varepsilon_1\neq\varepsilon_2)}\varepsilon_2(v,e)=\sum_{e\in
E(\varepsilon_2\neq\varepsilon_3)}\varepsilon_2(v,e)=0.
\]
Thus
\[
\begin{split}
\sum_{e\in E(\varepsilon_1\neq\varepsilon_3)}\varepsilon_1(v,e) &=
\sum_{e\in E(\varepsilon_1=\varepsilon_2\neq\varepsilon_3)}
\varepsilon_1(v,e)
+\sum_{e\in E(\varepsilon_1\neq\varepsilon_2=\varepsilon_3)}
\varepsilon_1(v,e)\\
&=\sum_{e\in E(\varepsilon_1=\varepsilon_2\neq\varepsilon_3)}
\varepsilon_2(v,e) -\sum_{e\in
E(\varepsilon_1\neq\varepsilon_2=\varepsilon_3)}
\varepsilon_2(v,e)\\
&= \sum_{e\in E(\varepsilon_1=\varepsilon_2\neq\varepsilon_3)\sqcup
E(\varepsilon_1\neq\varepsilon_2\neq\varepsilon_3)}
\varepsilon_2(v,e) \\
& \hspace{3ex}-\sum_{e\in
E(\varepsilon_1\neq\varepsilon_2=\varepsilon_3)\sqcup
E(\varepsilon_1\neq\varepsilon_2\neq\varepsilon_3)}
\varepsilon_2(v,e)\\
&= \sum_{e\in E(\varepsilon_2\neq\varepsilon_3)} \varepsilon_2(v,e)
-\sum_{e\in E(\varepsilon_1\neq\varepsilon_2)} \varepsilon_2(v,e)=0.
\end{split}
\]
This means that $\varepsilon_1$ is Eulerian equivalent to
$\varepsilon_3$.

(b) Suppose $(G,\rho)$ contains a directed cut $(U,\varepsilon_U)$,
where $U=[S,S^c]$ with $S\subseteq V$. Since $E(\rho\neq\sigma)$ is
directed Eulerian with the orientation $\rho$, then by
Proposition~\ref{Characterization-of-Eulerian}(b),
\[
\sum_{x\in U\cap E(\rho\neq\sigma)}[\rho,\varepsilon_U](x)=0.
\]
Since $[\rho,\varepsilon_U]\equiv 1$, it follows that $U\cap
E(\rho\neq\sigma)=\emptyset$. This means that
$\rho(v,e)=\sigma(v,e)$ for all edges $e\in U$, where $v$ is an
end-vertex of $e$ and $v\in S$. So $(U,\varepsilon_U)$ is a directed
cut of $(G,\varepsilon)$. This is a contradiction.

(c) For a flow $f\in q\bar\Delta^+_{\FL}(G,\sigma)$ ($f\in
q\Delta^+_{\FL}(G,\sigma)$), we have
\begin{eqnarray*}
\sum_{e\in E(v)}\rho(v,e)(Q_{\rho,\sigma}f)(e) &=& \sum_{e\in
E(v)}\sigma(v,e)f(e) +q\sum_{e\in E(\sigma\neq\rho)} \rho(v,e)=0.
\end{eqnarray*}
This shows that $Q_{\rho,\sigma}f\in q\bar\Delta_{\FL}^+(G,\rho)$
($Q_{\rho,\sigma}f\in q\Delta_{\FL}^+(G,\rho)$). Clearly,
$Q_{\rho,\sigma}$ sends lattice points to lattice points by
definition. Therefore,
$\bar\varphi_{\rho}(G,q)=\bar\varphi_{\sigma}(G,q)$ and
$\varphi_{\rho}(G,q)=\varphi_{\sigma}(G,q)$.
\end{proof}

For two Eulerian equivalent orientations $\rho,\sigma$, we have seen
that the digraph $(G,\rho)$ contains no directed cut if and only if
$(G,\sigma)$ contains no directed cut. This means that $\sim$
induces an equivalence relation on ${\mathcal O}_\textsc{tc}(G)$;
and each equivalence class of $\sim$ in ${\mathcal
O}_\textsc{tc}(G)$ is an equivalence class of $\sim$ in $\mathcal
O(G)$. We denote by $[{\mathcal O}_\textsc{tc}(G)]$ the quotient set
${\mathcal O}_\textsc{tc}(G)/\!\!\sim$ of Eulerian equivalence
classes. For each $\rho\in {\mathcal O}_\textsc{tc}(G)$, let
$[\rho]\in[{\mathcal O}_\textsc{tc}(G)]$ denote the equivalence
class of $\rho$, and define
\[
\varphi_{[\rho]}\bigl(G,q\bigr)=\varphi_{\rho}(G,q) \sp\mbox{and}\sp
\bar\varphi_{[\rho]}\bigl(G,q\bigr)=\bar\varphi_{\rho}(G,q).
\]
The following nontrivial Lemma is due to Tutte. It is crucial to the
proof of our main result Theorem~\ref{Modular-Flow-Thm}; so we
present a proof with our notations.

\begin{lemma}[Tutte \cite{Tutte1}]\label{Surjective}
The map $\Mod_q: F_{\mathbbm z}(G,\varepsilon;q)\rightarrow
F(G,\varepsilon;{\Bbb Z}/q{\Bbb Z})$ and its restriction $\Mod_q:
F_{\rm nz\mathbbm z}(G,\varepsilon;q)\rightarrow F_{\rm
nz}(G,\varepsilon;{\Bbb Z}/q{\Bbb Z})$ are surjective.
\end{lemma}
\begin{proof}
The second part of the lemma implies the first part. In fact, every
flow $f\in F(G,\varepsilon;{\Bbb Z}q{\Bbb Z})$ can be viewed as a
nowhere-zero flow $f|_{E_f}$ on the subdigraph
$(V,E_f,\varepsilon)$, where $E_f:=\{e\in E\:|\: f(e)\neq 0\}$. Let
$g$ be a nowhere-zero integer $q$-flow on $(V,E_f,\varepsilon)$ such
that $\Mod_q(g)=f|_{E_f}$. Then $g$ is extended to an integer
$q$-flow on $(G,\varepsilon)$ by setting $g\equiv1$ on $E-E_f$.

Now for each $f\in F_{\mathbbm z}(G,\varepsilon;q)$ we write $\tilde
f=\Mod_qf$. We identify ${\Bbb Z}/q{\Bbb Z}$ with the set
$\{0,1,\ldots,q-1\}$ and view each modular flow $g\in
F(G,\varepsilon;{\Bbb Z}/q{\Bbb Z})$ as an integer-valued function
$g:E\rightarrow\{0,1,\ldots,q-1\}$. Then the map
$Q_{\rho,\varepsilon}$ maps $F(G,\varepsilon;{\Bbb Z}/q{\Bbb Z})$ to
$F(G,\rho;{\Bbb Z}/q{\Bbb Z})$, and $F_{\rm nz\mathbbm
z}(G,\varepsilon;{\Bbb Z}/q{\Bbb Z})$ to $F_{\rm nz\mathbbm
z}(G,\rho;{\Bbb Z}/q{\Bbb Z})$. For each $g\in[0,q]^E$ and
$\rho\in{\mathcal O}(G)$, we define
\[
\eta(g,\rho):=\sum_{v\in V} \Big|\sum_{e\in E}\rho(v,e)g(e)\Big|.
\]
Fix a modular flow $\tilde f\in F_{\rm nz\mathbbm
z}(G,\varepsilon;{\Bbb Z}/q{\Bbb Z})$. Let $\rho^*$ be a particular
orientation on $G$ such that
\[
\eta(Q_{\rho^*,\varepsilon}\tilde f,\rho^*) =\min\{
\eta(Q_{\rho,\varepsilon}\tilde f,\rho)\:|\: \rho\in{\mathcal
O}(G)\}.
\]
We write $f^*:=Q_{\rho^*,\varepsilon}\tilde f$ and define
$f:=P_{\varepsilon,\rho^*}f^*$.

If $\eta(f^*,\rho^*)=0$, then $\sum_{e\in E}\rho^*(v,e)f^*(e)=0$ for
all $v\in V$. This means that $f^*$ is an integer $q$-flow of
$(G,\rho^*)$. Whence $f$ is an integer $q$-flow of
$(G,\varepsilon)$. By definition of $P_{\varepsilon,\rho^*}$ and
$Q_{\rho^*,\varepsilon}$, we see that
\begin{eqnarray*}
f(e) &=& \left\{
\begin{array}{ll}
\tilde f(e) & \mbox{if $\varepsilon(v,e)=\rho^*(v,e)$}\\
\tilde f(e)-q & \mbox{if $\varepsilon(v,e)\neq\rho^*(v,e)$}
\end{array}\right.\\ \\
&=& \tilde f(e)\;(\mod q) \sp \mbox{for all $e\in E$.}
\end{eqnarray*}
The surjectivity of $\Mod_q: F_{\rm nz\mathbbm
z}(G,\varepsilon;q)\rightarrow F_{\rm nz}(G,\varepsilon;{\Bbb
Z}/q{\Bbb Z})$ follows immediately.

We now claim that $\eta(f^*,\rho^*)=0$. Suppose
$\eta(f^*,\rho^*)>0$. Then there exist a vertex $u\in V$ and an
integer $k$ such that
\begin{equation}\label{RF*}
\sum_{e\in E}\rho^*(u,e)f^*(e)=kq >0\hspace{1ex} (<0).
\end{equation}
Note that for any function $g:E\rightarrow A$, where $A$ is an
abelian group, and for any orientation $\rho\in\mathcal O(G)$, we
have
\[
\sum_{v\in V} \sum_{e\in E} \rho(v,e) g(e)=0.
\]
In particular, for the function $f^*$ and the orientation $\rho^*$,
we have
\[
\sum_{v\in V} \sum_{e\in E} \rho^*(v,e) f^*(e)=0.
\]
Since (\ref{RF*}) and $f^*$ is a modular flow, there exists a vertex
$w$ such that
\[
\sum_{e\in E}\rho^*(w,e)f^*(e)=-lq<0\hspace{1ex} (>0).
\]
Thus there is a path $P=v_0e_1v_1\cdots e_nv_n$ with $v_0=u$ and
$v_n=w$, such that $\rho^*(v_0,e_1)=1$, $\rho^*(v_n,e_n)=-1$, and
$\rho^*(v_i,e_i)\rho^*(v_i,e_{i+1})=-1$, where $1\leq i\leq n-1$.
Let $\rho^{**}$ be an orientation on $G$ given by
\[
\rho^{**}(v,e)=\left\{
\begin{array}{rl}
-\rho^*(v,e) & \mbox{if $e=e_i$ for some $1\leq i\leq n$,}\\
 \rho^*(v,e)  & \mbox{otherwise.}
\end{array}\right.
\]
We write $f^{**}:=Q_{\rho^{**},\rho^*}f^*$. Then
$f^{**}=Q_{\rho^{**},\rho^*}Q_{\rho^*,\varepsilon}\tilde
f=Q_{\rho^{**},\varepsilon}\tilde f$.

Notice that at each vertex $v\in V$, we have
\[
\sum_{e\in E}\rho^{**}(v,e)f^{**}(e)= \sum_{e\in E}\rho^*(v,e)f^*(e)
-q\sum_{e\in E(\rho^*\neq\rho^{**})} \rho^{*}(v,e).
\]
In particular, for the vertices $u,w$, and other vertices $v$
different from $u$ and $w$, we have
\[
\sum_{e\in E}\rho^{**}(u,e)f^{**}(e)= \sum_{e\in E}\rho^*(u,e)f^*(e)
-q=(k-1)q,
\]
\[
\sum_{e\in E}\rho^{**}(w,e)f^{**}(e)= \sum_{e\in E}\rho^*(w,e)f^*(e)
+q=(1-l)q,
\]
\[
\sum_{e\in E}\rho^{**}(v,e)f^{**}(e)= \sum_{e\in
E}\rho^*(v,e)f^*(e).
\]
It follows that
\[
\eta(f^{**},\rho^{**})=\eta(f^*,\rho^*)-2q<\eta(f^*,\rho^*).
\]
This is contradictory to the selection of $\rho^*$ that
$\eta(f^*,\rho^*)$ is minimum.
\end{proof}\vspace{1ex}

For each real-valued function $f:E\rightarrow{\Bbb R}$ and any
orientation $\rho$ on $G$, we associate with $f$ and $\rho$ an
orientation $\rho_f$, defined for each $(v,e)\in V\times E$ by
\begin{equation}\label{DF}
\rho_f(v,e)=\left\{
\begin{array}{rl}
\rho(v,e)  & \mbox{if $f(e)>0$,}\\
-\rho(v,e) & \mbox{if $f(e)\leq0$}.
\end{array}\right.
\end{equation}
For two orientations $\rho,\sigma\in{\mathcal O}(G)$, we associate a
{\em symmetric difference function}
$I_{\rho,\sigma}:E\rightarrow\{0,1\}$, defined for each edge $e\in
E$ (and its one end-vertex $v$) by
\begin{equation}\label{FD}
I_{\rho,\sigma}(e)=\left\{
\begin{array}{rl}
0  & \mbox{if $\rho(v,e)=\sigma(v,e)$,}\\
1 & \mbox{if $\rho(v,e)\neq\sigma(v,e)$.}
\end{array}\right.
\end{equation}

\begin{lemma}\label{Eulerian}
Let $f_1,f_2\in F(G,\varepsilon;q)$. If $f_1(e)\equiv f_2(e)\:(\mod
q)$ for all edges $e\in E$, then  $\varepsilon_{f_1}$ and
$\varepsilon_{f_2}$ are Eulerian equivalent.
\end{lemma}
\begin{proof}
Let us simply write $\varepsilon_i:=\varepsilon_{f_i}$, $i=1,2$. It
suffices to show that at each vertex $v$,
\begin{equation}\label{E1E2}
\sum_{e\in E(\varepsilon_{1}\neq\varepsilon_{2})}
\varepsilon_2(v,e)=0.
\end{equation}
Since $f_1(e)\equiv f_2(e)$ $(\mod q)$ for all $e\in E$, then
\[
f_1(e)=f_2(e)+a_eq,\sp e\in E,
\]
where $a_e\in\{-1,0,1\}$. More precisely, for each edge $e$ at its
one end-vertex $v$,
\[
f_1(e)=\left\{\begin{array}{ll}
 f_2(e) & \mbox{if $\varepsilon_1(v,e)=\varepsilon_2(v,e)$,}\\
 f_2(e)-q & \mbox{if $\varepsilon_1(v,e)\neq\varepsilon_2(v,e)=\varepsilon(v,e)$,}\\
 f_2(e)+q & \mbox{if
 $\varepsilon_1(v,e)\neq \varepsilon_2(v,e)\neq\varepsilon(v,e)$.}
\end{array}\right.
\]
Let $g_i=P_{\varepsilon_i,\varepsilon}f_i$, $i=1,2$. Then $g_i$ are
real $q$-flows of $(G,\varepsilon_i)$, and for an edge $e$ at its
one end-vertex $v$,
\[
g_1(e)=\left\{\begin{array}{rl}
g_2(e)   & \mbox{if $\varepsilon_1(v,e)=\varepsilon_2(v,e)$,}\\
q-g_2(e) & \mbox{if $\varepsilon_1(v,e)\neq\varepsilon_2(v,e)$.}
\end{array}\right.
\]
Note that at each vertex $v\in V$,
\begin{equation}\label{g12}
\sum_{e\in E}\varepsilon_i(v,e)g_i(e)=0, \sp i=1,2.
\end{equation}
Now consider the case of $i=1$ in (\ref{g12}). Replace $g_1(e)$ by
$g_2(e)$ if $\varepsilon_1(v,e)=\varepsilon_2(v,e)$ and by
$q-g_2(e)$ if $\varepsilon_1(v,e)\neq\varepsilon_2(v,e)$; we obtain
\begin{eqnarray*}
\lefteqn{\sum_{e\in E(\varepsilon_1=\varepsilon_2)}
\varepsilon_2(v,e)g_2(e)
-\sum_{e\in E(\varepsilon_1\neq\varepsilon_2)} \varepsilon_2(v,e)(q-g_2(e))}\\
&& \hspace{5mm} =\sum_{e\in E} \varepsilon_2(v,e)g_2(e) -q\sum_{e\in
E(\varepsilon_1\neq\varepsilon_2)} \varepsilon_2(v,e)=0.
\end{eqnarray*}
Applying (\ref{g12}) for $i=2$, we see that (\ref{E1E2}) is true.
\end{proof}

\begin{lemma}\label{Minverse}
Let $\rho,\sigma,\omega\in {\mathcal O}_\textsc{tc}(G)$ be Eulerian
equivalent orientations, and let $f\in
q\Delta^\rho_{\FL}(G,\varepsilon)$ be a real $q$-flow. Then
\begin{enumerate}
\item[(a)] $\varepsilon_f=\rho$.

\item[(b)] $P_{\varepsilon,\sigma} Q_{\sigma,\rho} P_{\rho,\varepsilon}
(q\Delta^\rho_{\FL}(G,\varepsilon))
=q\Delta^\sigma_{\FL}(G,\varepsilon)$.

\item[(c)] $P_{\varepsilon,\sigma} Q_{\sigma,\rho}
P_{\rho,\varepsilon}f=P_{\varepsilon,\omega} Q_{\omega,\rho}
P_{\rho,\varepsilon}f$ if and only if $\sigma=\omega$.

\item[(d)] $F(G,\varepsilon;q)\cap \Mod_q^{-1}\bigl(\Mod_q f\bigr) =
\{P_{\varepsilon,\alpha} Q_{\alpha,\rho} P_{\rho,\varepsilon}f\:|\:
\alpha\sim\rho\}$.
\end{enumerate}
\end{lemma}
\begin{proof}
(a) By definition of $\Delta^\rho_\FL(G,\varepsilon)$ and the fact
$f\in q\Delta^\rho_\FL(G,\varepsilon)$, we have
$[\rho,\varepsilon](e)f(e)>0$ for all $e\in E$. So for each edge $e$
at its one end-vertex $v$, $\rho(v,e)=\varepsilon(v,e)$ if $f(e)>0$,
and $\rho(v,e)=-\varepsilon(v,e)$ if $f(e)<0$. By definition of
$\varepsilon_f$, we see that $\varepsilon_f=\rho$.

(b) Recall Lemma~\ref{PRE} and Lemma~\ref{QRE}(c) that
\[
P_{\rho,\varepsilon} (q\Delta^\rho_{\FL}(G,\varepsilon))
=q\Delta^+_{\FL}(G,\rho) \sp\mbox{and}\sp
Q_{\sigma,\rho}(q\Delta^+_{\FL}(G,\rho)\bigr)=q\Delta^+_{\FL}(G,\sigma).
\]
Since $P_{\varepsilon,\sigma}$ is an involution, then
$P_{\varepsilon,\sigma}(q\Delta^+_{\FL}(G,\sigma))
=q\Delta^\sigma_{\FL}(G,\varepsilon)$. The identity follows
immediately by composition.

(c) The sufficiency is trivial. For necessity, we write
$g:=P_{\varepsilon,\sigma} Q_{\sigma,\rho} P_{\rho,\varepsilon}f$
and $h:=P_{\varepsilon,\omega} Q_{\omega,\rho}
P_{\rho,\varepsilon}f$. Since $g\in
q\Delta^\sigma_{\FL}(G,\varepsilon)$ and $h\in
q\Delta^\omega_{\FL}(G,\varepsilon)$, we have $\varepsilon_g=\sigma$
and $\varepsilon_h=\omega$. Clearly, if $g=h$, then $\sigma=\omega$.

(d) Let us write $g:=P_{\varepsilon,\alpha} Q_{\alpha,\rho}
P_{\rho,\varepsilon}f$ for an orientation $\alpha$ such that
$\alpha\sim\rho$. By definition of $P_{\varepsilon,\alpha},
Q_{\alpha,\rho}, P_{\rho,\varepsilon}$, we have
$g\in\Delta^\alpha_\FL(G,\varepsilon)$ and
\[
g(e)=f(e)+a_eq,\sp e\in E,
\]
where $a_e\in\{-1,0,1\}$. Clearly, $g(e)\equiv f(e)(\mod q)$ for all
$e\in E$; namely, $g\in\Mod_q^{-1}\big(\Mod_qf\big)$. Conversely,
let $h\in F(G,\varepsilon;q)$ be such that $\Mod_qh=\Mod_qf$. We
have
\[
h(e)=f(e)+b_eq, \sp e\in E,
\]
where $b_e\in\{-1,0,1\}$. By definition of
$\varepsilon_f,\varepsilon_h$, a straightforward calculation shows
that for each non-loop edge $e$ (and its one end-vertex $v$),
\[
h(e) =\left\{
\begin{array}{ll}
f(e) & \mbox{if $h(e)>0,f(e)>0$ ($\Leftrightarrow$
$\varepsilon_h(v,e)=\varepsilon_f(v,e)=\varepsilon(v,e)$}),\\
f(e) & \mbox{if $h(e)\leq 0,f(e)\leq 0$ ($\Leftrightarrow$
$\varepsilon_h(v,e)=\varepsilon_f(v,e)\neq \varepsilon(v,e)$}),\\
f(e)-q & \mbox{if $h(e)\leq 0,f(e)>0$ ($\Leftrightarrow$
$\varepsilon_h(v,e)\neq\varepsilon_f(v,e)=\varepsilon(v,e)$),}\\
f(e)+q & \mbox{if $h(e)>0,f(e)\leq 0$ ($\Leftrightarrow$
$\varepsilon_h(v,e)\neq\varepsilon_f(v,e)\neq\varepsilon(v,e)$).}
\end{array}\right.
\]
By definition of $P_{\varepsilon,\varepsilon_h},
Q_{\varepsilon_h,\varepsilon_f}, P_{\varepsilon_f,\varepsilon}$,
another straightforward calculation shows that for each non-loop
edge $e$ (and its one end-vertex $v$),
\[
(P_{\varepsilon,\varepsilon_h}Q_{\varepsilon_h,\varepsilon_f}P_{\varepsilon_f,\varepsilon}f)(e)
=\left\{
\begin{array}{ll}
f(e) & \mbox{if $\varepsilon_h(v,e)=\varepsilon_f(v,e)=\varepsilon(v,e)$,}\\
f(e) & \mbox{if $\varepsilon_h(v,e)=\varepsilon_f(v,e)\neq\varepsilon(v,e)$,}\\
f(e)-q & \mbox{if $\varepsilon_h(v,e)\neq\varepsilon_f(v,e)=\varepsilon(v,e)$,}\\
f(e)+q & \mbox{if
$\varepsilon_h(v,e)\neq\varepsilon_f(v,e)\neq\varepsilon(v,e)$.}
\end{array}\right.
\]
This means that $h=P_{\varepsilon,\varepsilon_h}
Q_{\varepsilon_h,\varepsilon_f} P_{\varepsilon_f,\varepsilon}f$.
Since $\varepsilon_f=\rho$ by Part (a), thus
$h=P_{\varepsilon,\alpha} Q_{\alpha,\rho} P_{\rho,\varepsilon}f$
with $\alpha=\varepsilon_h$.
\end{proof}

\begin{prop}\label{EER}
The number of orientations on $G$ that are Eulerian equivalent to
$\varepsilon$ is the number of $0$-$1$ flows of $(G,\varepsilon)$,
i.e.,
\begin{equation}\label{Cardinality-Eulerian-equivalence-class}
 \#[\varepsilon] = \bar{\varphi}_{\varepsilon}(G,1) = |\bar\Delta^+_\FL(G,\varepsilon)\cap{\Bbb Z}^E|.
\end{equation}
\end{prop}
\begin{proof}
It is enough to show that the following map
\[
[\varepsilon]=\{\rho\in{\mathcal
O}(G)\:|\:\rho\sim\varepsilon\}\rightarrow
\bar\Delta^+_\FL(G,\varepsilon)\cap{\Bbb Z}^E,\sp \rho\mapsto
I_{\rho,\varepsilon},
\]
is a bijection. Fix an orientation $\rho$ that is Eulerian
equivalent to $\varepsilon$. Note that $I_{\rho,\varepsilon}$ is the
characteristic function of the edge subset $E(\rho\neq\varepsilon)$.
Since $E(\rho\neq\varepsilon)$ is directed Eulerian with the
orientation $\varepsilon$. Then $I_{\rho,\varepsilon}$ is a flow of
$(G,\varepsilon)$ with values in $\{0,1\}$. So the map $\rho\mapsto
I_{\rho,\varepsilon}$ is well-defined, and is clearly injective.
Conversely, given a 0-1 flow $f$ of $(G,\varepsilon)$. Let $\rho$ be
an orientation on $G$ defined by $\rho(v,e)=\varepsilon(v,e)$ if
$f(e)=0$ and $\rho(v,e)=-\varepsilon(v,e)$ if $f(e)=1$, where $v$ is
an end-vertex of the edge $e$. Then $E(\rho\neq\varepsilon)=\{e\in
E(G)\:|\: f(e)=1\}$, which is directed Eulerian with the orientation
$\varepsilon$. This means that $\rho\sim\varepsilon$ and
$I_{\rho,\varepsilon}=f$. Hence the map is surjective.
\end{proof}
\vspace{1ex}

{\sc Proof of Theorem~\ref{Modular-Flow-Thm}.} \vspace{2ex}

It has been shown that $\varphi(G,q)$ is a polynomial function of
degree $n(G)$. Fix an orientation $\rho\in {\mathcal
O}_\textsc{tc}(G)$. For each orientation $\sigma\sim\rho$ and any
$f\in q\Delta^\sigma_{\FL}(G,\varepsilon)$, we have
\begin{eqnarray}
|F(G,\varepsilon;q)\cap\Mod_q^{-1}(\Mod_qf)| &=&
\#\{P_{\varepsilon,\alpha}Q_{\alpha,\rho}P_{\rho,\varepsilon}f \:|\:
\alpha\sim\sigma\} \nonumber\\
&=& \#[\sigma]=\#[\rho]. \label{Inverse}
\end{eqnarray}

Now apply Parts (b) and (d) of Lemma~\ref{Minverse}; we have the
disjoint unions
\begin{align}
\bigsqcup_{\sigma\in[\rho]} q\Delta^\sigma_{\FL}(G,\varepsilon) &=
\bigsqcup_{\sigma\in[\rho]} P_{\varepsilon,\sigma} Q_{\sigma,\rho}
P_{\rho,\varepsilon} (q\Delta^\rho_{\FL}(G,\varepsilon)) \nonumber\\
&= \bigsqcup_{\sigma\in[\rho],\: f\in
q\Delta^\rho_\FL(G,\varepsilon)} \{P_{\varepsilon,\sigma}
Q_{\sigma,\rho}
P_{\rho,\varepsilon}f\} \nonumber\\
&= \bigsqcup_{f\in q\Delta^\rho_\FL(G,\varepsilon)}
F(G,\varepsilon;q)\cap \Mod_q^{-1}(\Mod_qf) \nonumber \\
&= F(G,\varepsilon;q)\cap
\Mod_q^{-1}\Mod_q(q\Delta^\rho_{\FL}(G,\varepsilon)). \nonumber
\end{align}
Note that the above orientation $\rho$ can be replaced by any
orientation $\sigma$ that is Eulerian equivalent to $\rho$. We
further have
\begin{equation}\label{ABCD}
\bigsqcup_{\sigma\in[\rho]} q\Delta^\sigma_{\FL}(G,\varepsilon) =
F(G,\varepsilon;q)\cap
\Mod_q^{-1}\Mod_q\bigg(\bigsqcup_{\sigma\in[\rho]}
q\Delta^\sigma_{\FL}(G,\varepsilon)\bigg).
\end{equation}
On the one hand, recall $\varphi_{\sigma}(G,q)=\varphi_{\rho}(G,q)$
whenever $\sigma\sim\rho$ (see Lemma~\ref{QRE}); then the number of
lattice points in the left-hand side of \eqref{ABCD} is
\[
\varphi_{\rho}(G,q)\cdot \#[\rho].
\]
On the other hand, \eqref{Inverse} implies that the number of
lattice points of the right-hand side of \eqref{ABCD} is
\[
\Big|\Mod_q\Big(\bigsqcup_{\sigma\in[\rho]}
q\Delta^\sigma_{\FL}(G,\varepsilon) \cap {\Bbb Z}^E\Big)\Big|\cdot
\#[\rho].
\]
It then follows that
\[
\Bigl|\Mod_q\Bigl(\bigsqcup_{\sigma\in[\rho]}
q\Delta^\sigma_{\FL}(G,\varepsilon) \cap {\Bbb Z}^E \Bigr)\Bigr| =
\varphi_{\rho}(G,q).
\]
Note that (\ref{Flow-Decomposition}) implies the disjoint
decomposition
\[
F_{\rm nz{\mathbbm z}}(G,\varepsilon;q) = \bigsqcup_{\rho\in
[{\mathcal O}_\textsc{tc}(G)]} \bigsqcup_{\sigma\in[\rho]}
q\Delta^\sigma_{\FL}(G,\varepsilon)\cap{\Bbb Z}^E.
\]
Applying the map $\Mod_q$, we obtain the following disjoint
decomposition
\begin{equation}\label{MD}
\begin{split}
F_{\rm nz}(G,\varepsilon;{\Bbb Z}/q{\Bbb Z}) &= \Mod_q(F_{\rm
nz{\mathbbm z}}(G,\varepsilon;q)) \\
&= \bigsqcup_{[\rho]\in[{\mathcal O}_\textsc{tc}(G)]}
\Mod_q\bigg(\bigsqcup_{\sigma\in[\rho]}
q\Delta^\sigma_{\FL}(G,\varepsilon)\cap{\Bbb Z}^E\bigg).
\end{split}
\end{equation}
The first equation follows from the surjectivity of $\Mod_q$ by
Lemma~\ref{Surjective}. To see the disjointness of the union in the
second equation, suppose the union is not disjoint. This means that
there exist integer flows $f_i\in
q\Delta^{\sigma_i}_\FL(G,\varepsilon)$ and orientations $\rho_i$
such that $\sigma_i\sim\rho_i$ ($i=1,2$), $[\rho_1]\neq[\rho_2]$,
and $\Mod_q(f_1)=\Mod_q(f_2)$. Then Lemma~\ref{Minverse}(a) implies
that $\varepsilon_{f_i}=\sigma_i$, and Lemma~\ref{Eulerian} implies
that $\varepsilon_{f_1}\sim \varepsilon_{f_2}$. It follows by
transitivity that $\rho_1\sim\rho_2$, i.e., $[\rho_1]=[\rho_2]$.
This is a contradiction.

Counting the number of elements of both sides of \eqref{MD}, we
obtain
\[
\varphi(G,q)=\sum_{[\rho]\in [{\mathcal O}_\textsc{tc}(G)]}
\varphi_{\rho}(G,q).
\]
The Reciprocity Law (\ref{Modular-Flow-Reciprocity-Law}) follows
from the Reciprocity Law (\ref{FPRL}) and the definition of
$\bar\varphi(G,q)$.

Let $q=0$. We have $\varphi(G,0)=(-1)^{n(G)}\bar\varphi(G,0)$ by
(\ref{Modular-Flow-Reciprocity-Law}). Since
$\bar\varphi_{\rho}(G,0)=1$ for all $\rho\in {\mathcal
O}_\textsc{tc}(G)$, we see that $\bar\varphi(G,0)=\#[{\mathcal
O}_\textsc{tc}(G)]$ by (\ref{Closed-Modular-Flow-Sum}). Hence
$\varphi(G,0)=(-1)^{n(G)}\#[{\mathcal O}_\textsc{tc}(G)]$.
\hfill{$\Box$}\\

Searching online we found a result on modular flow reciprocity by
Breuer and Sanyal \cite{Breuer-Sanyal1}, which states that
$(-1)^{n(G)}\varphi(G,-q)$ counts the numbers of pairs $(f,\sigma)$,
where $f$ is a flow of $(G,\varepsilon)$ modulo $q$ and $\sigma$ is
a totally cyclic reorientation of the digraph $(G/{\supp
f},\varepsilon)$. The result can be written as the sum
\begin{equation}\label{BS1}
\varphi(G,-q)=(-1)^{n(G)}\sum_{X\subseteq E}\varphi(\langle
X\rangle,q)\,|\mathcal{O}_\textsc{tc}(G/X)|,
\end{equation}
where $G/X$ is the graph obtained from $G$ by contacting the edges
of $X$. The term $\varphi(\langle X\rangle,q)\,|{\mathcal
O}_\textsc{tc}(G/X)|$ in (\ref{BS1}) is nonzero if and only if the
graphs $\langle X\rangle$ and $G/X$ are bridgeless. The formula
(\ref{BS1}) can be argued straightforward as follows.

Notice the trivial fact that each flow $f$ corresponds to a
nowhere-zero flow on its support $\supp(f):=\{e\in E\:|\:f(e)\neq
0\}$. This means that
\[
t^{n\langle X\rangle} = \sum_{Y\subseteq X} \varphi(\langle
Y\rangle,t),\sp X\subseteq E.
\]
The M\"{o}bius inversion implies
\[
\varphi(\langle X\rangle,t)=\sum_{Y\subseteq X}(-1)^{|X-Y|}
t^{n\langle Y\rangle},\sp X\subseteq E.
\]
In particular, for $X=E$ and $t=-q$, we have
\begin{eqnarray*}
\varphi(G,-q) &=& \sum_{Y\subseteq E} (-1)^{|E-Y|+n\langle
Y\rangle} q^{n\langle Y\rangle}\\
&=& \sum_{Y\subseteq E} (-1)^{|E-Y|+n\langle Y\rangle}
\sum_{X\subseteq Y}\varphi(\langle X\rangle,q)\\
&=& \sum_{X\subseteq E}\varphi(\langle X\rangle,q) \sum_{X\subseteq
Y\subseteq E}(-1)^{|E-Y|+n\langle Y\rangle}.
\end{eqnarray*}
Since $n\langle Y\rangle=n\langle X\rangle+n(\langle Y\rangle/X)$
for all edge subsets $Y$ of $G$ such that $X\subseteq Y$, we see
that for each fixed edge subset $X\subseteq E$,
\begin{eqnarray*}
\sum_{X\subseteq Y\subseteq E}(-1)^{|E-Y|+n\langle Y\rangle} &=&
\sum_{X\subseteq Y\subseteq E}(-1)^{|E-Y|+n\langle
X\rangle+n(\langle Y\rangle/X)} \nonumber\\
&=& (-1)^{n\langle X\rangle} \sum_{Z\subseteq E(G/X)}
(-1)^{|E(G/X)-Z|+n\langle Z\rangle} \nonumber\\
&=& (-1)^{n\langle X\rangle +n\langle G/X\rangle}
|\mathcal{O}_\textsc{tc}(G/X)|.
\end{eqnarray*}
The above last equality follows from the Zaslavsky formula
(\ref{Zaslavsky-First-Formula}) about the flow arrangement
${\mathcal A}_\FL(G,\varepsilon;{\Bbb R})$, i.e.,
\[
\varphi(G,-1)=\sum_{Z\subseteq E}(-1)^{|E|-|Z|+n\langle Z\rangle}=
(-1)^{n(G)}|\mathcal{O}_\textsc{tc}(G)|.
\]
Now the identity (\ref{BS1}) follows immediately.

\section{Connection with the Tutte polynomial}

The Tutte polynomial (see \cite{Bollobas1}, p.337) of a graph
$G=(V,E)$ is a polynomial in two variables, which may be defined as
\begin{gather}
T_G(x,y)=\sum_{A\subseteq E}(x-1)^{r\langle E\rangle-r\langle
A\rangle}(y-1)^{n\langle A\rangle},
\end{gather}
where $\langle E\rangle=|V|-c(G)$, $r\langle A\rangle=|V|-c\langle
A\rangle$, $n\langle A\rangle=|A|-r\langle A\rangle$. The polynomial
$T_G(x,y)$ satisfies the Deletion-Contraction Relation:
\[
T_G(x,y)=\left\{\begin{array}{ll} xT_{G/e}(x,y) &\mbox{if $e$ is a bridge}, \\
yT_{G-e}(x,y) & \mbox{if $e$ is a loop}, \\
T_{G-e}(x,y)+T_{G/e}(x,y) & \mbox{otherwise}.
\end{array}\right.
\]
It is well-known that the flow polynomial $\varphi(G,t)$ is related
to $T_G(x,y)$ by
\begin{gather}
\varphi(G,t)=(-1)^{n(G)} T_G(0,1-t).
\end{gather}
Thus
\[
\bar\varphi(G,t)=(-1)^{n(G)}\varphi(G,-t)=T_G(0,t+1).
\]
We conclude the information as the following proposition.

\begin{prop}
The Tutte polynomial $T_G$ is related to $\varphi$ and $\bar\varphi$
as follows:
\begin{align}
T_G(0,t) &= \bar\varphi(G,t-1)=(-1)^{n(G)}\varphi(G,1-t).
\end{align}
In particular, $T_G(0,1)=|\varphi(G,0)|=\bar\varphi(G,0)$ counts the
number of Eulerian equivalence classes of totally cyclic
orientations on $G$.
\end{prop}

\begin{exmp}\label{BNUV}
Let $B_n(u,v)$ be a graph with two vertices $u,v$, and $n$ multiple
edges $e_1,\dots,e_n$  between $u$ and $v$. Let $\varepsilon$ be an
orientation of $B_n(u,v)$ such that all edges have the direction
from $u$ to $v$. The number of integer flows $f$ of $B_n(u,v)$ such
that $|f|<q$ is equal to the number of integer solutions of the
linear inequality system
\[
x_1+\cdots+x_n=0,\sp -(q-1)\leq x_i\leq q-1.
\]
Let $x_i=y_i-(q-1)$. The above system reduces to
\begin{equation}\label{Y2A}
y_1+\cdots+y_n=n(q-1),\sp 0\leq y_i\leq 2q-2.
\end{equation}
Recall that the number of nonnegative integer solutions of
$y_1+\cdots+y_n=r$ is
\[
\Bigl\langle{n\atop r}\Bigr\rangle:=\Bigl({n+r-1\atop
r}\Bigr)=\Bigl({n+r-1\atop n-1}\Bigr).
\]
So the number of nonnegative integer solutions of
$y_1+\cdots+y_n=n(q-1)$ is $\bigl({nq-1\atop n-1}\bigr)$, which
includes the number of integer solutions of (\ref{Y2A}) and the
number of nonnegative integer solutions having at least one $y_i\geq
2q-1$.

Let $s_n(q)$ denote the number of integer solutions of (\ref{Y2A}).
To figure out $s_n(q)$, let $Y$ be the set of nonnegative integer
solutions of $y_1+\cdots+y_n=n(q-1)$, $Y_i$ the set of nonnegative
integer solutions of $y_1+\cdots+y_n=n(q-1)$ with the $i$th variable
$y_i\geq 2q-1$, and $Y_0$ the set of integer solutions of
(\ref{Y2A}). Then $Y_0=Y-\bigcup_{i=1}^nY_i$. Consider the case
where $j$ variables are greater than or equal to $2q-1$, say,
$y_{n-j+1},\ldots,y_n$; then $y_1+\cdots+y_{n-j}\leq
n(q-1)-j(2q-1)$. The number of such solutions is equal to the number
of nonnegative integer solutions of the equation
$y_1+\cdots+y_{n-j}+y'=n(q-1)-j(2q-1)$, which is given by
\[
\biggl\langle {n-j+1\atop n(q-1)-j(2q-1)}
\biggr\rangle=\left({(n-2j)q\atop
n(q-1)-j(2q-1)}\right)=\left({(n-2j)q\atop n-j}\right).
\]
Applying the Inclusion-Exclusion Principle,
\begin{eqnarray*}
s_n(q) &=& \#(Y)+\sum_{\emptyset\neq I\subseteq[n]}(-1)^{|I|} \#\bigg(\bigcap_{i\in I} Y_i\bigg)\\
&=& \biggl({nq-1\atop n-1}\biggr)
+\sum_{j=1}^{\lfloor\frac{n(q-1)}{2q-1}\rfloor} (-1)^j \left({n\atop
j}\right) \left({(n-2j)q\atop n-j}\right).
\end{eqnarray*}
We list $s_1,s_2,s_3,s_4$ explicitly as follows:
\[
s_1(q)=1,\sp s_2(q)=2q-1,\sp s_3(q)=3q^2-3q+1,
\]
\[
s_4(q)=\left({4q-1\atop 3}\right)-4\left({2q\atop 3}\right)\\
= \frac{1}{3}(2q-1)\left(8q^2-8q+3\right).
\]
Let $s_0(q)$ $(\equiv 1)$ denote the number of zero flows of
$B_n(u,v)$. Then $\varphi_{\mathbbm z}(B_4,q)$, defined as the
number of nowhere-zero integer $q$-flows of $B_4(u,v)$, counts the
number of integer flows $f$ such that $0<|f|<q$, and is given by
\begin{eqnarray}
\varphi_{\mathbbm z}(B_4,q) &=&
\#\biggl(\bigl[1-q,q-1\bigr]^{E(B_4)}\cap F(B_4,\varepsilon;{\Bbb
Z})- \bigcup_{e\in E(B_4)} F_e\biggr) \\
&=& s_4(q)-4s_3(q)+6s_2(q)-4s_1(q)+s_0(q) \nonumber\\
&=& \frac{2}{3}(q-1) \bigl(8q^2-22q+21\bigr). \nonumber
\end{eqnarray}

Likewise, the number of flows of $B_n(u,v)$ modulo $q$ is $q^{n-1}$,
$n\geq 1$. Thus $\varphi(B_4,q)$, defined as the number of
nowhere-zero flows of $B_4(u,v)$ modulo $q$, is given by
\begin{eqnarray}
\varphi(B_4,q) &=& \#\biggl(F(B_4,\varepsilon;{\Bbb
Z}/q{\Bbb Z})- \bigcup_{e\in E(B_4)} F_e\biggr) \\
&=& q^3-4q^2+6q-4+1 \nonumber\\
&=& (q-1)\bigl(q^2-3q+3\bigr). \nonumber
\end{eqnarray}
In particular,
\[
|\varphi_{\mathbbm z}(B_4,0)\bigr|=\bigl|\varphi(B_4,-1)|=14, \sp
|\varphi(B_4,0)|=3.
\]
There are $14$ totally cyclic orientations on $B_4(u,v)$ by
Theorem~\ref{Integral-Flow-Thm}. The $14$ orientations can be
grouped into $3$ Eulerian equivalence classes by
Theorem~\ref{Modular-Flow-Thm}; see Figure~\ref{B4}.
\begin{figure}[h]
\centering
\includegraphics[width=50mm]{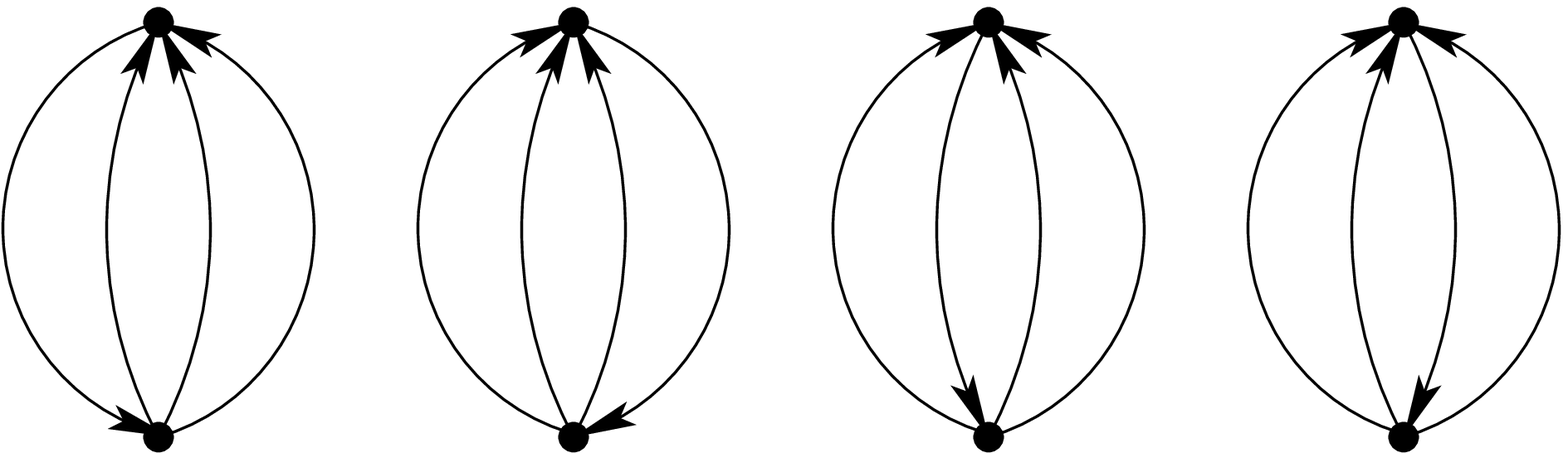}
\vspace{2ex}

\includegraphics[width=50mm]{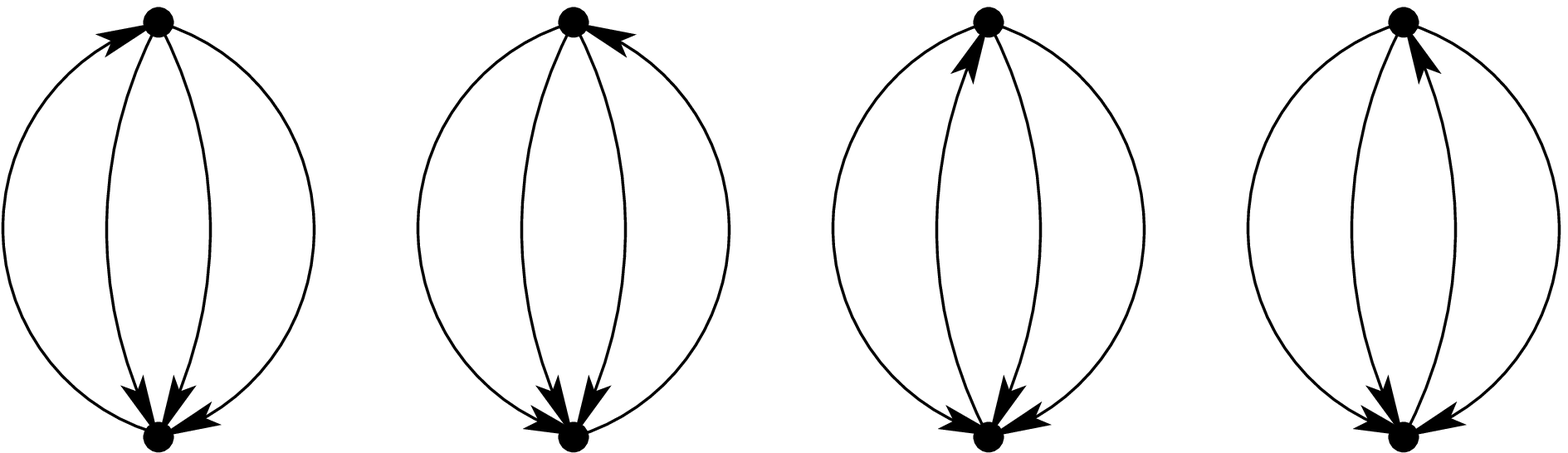}
\vspace{2ex}

\includegraphics[width=75mm]{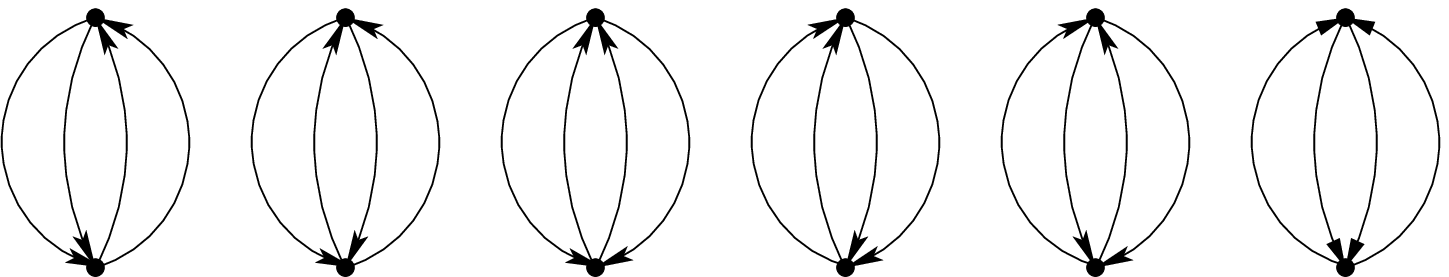}
\caption{The three Eulerian
equivalence classes of ${\mathcal O}_\textsc{tc}(B_4)$. \label{B4}}
\end{figure}

Let $\varepsilon_k$ be an orientation on $B_n(u,v)$ such that the
arrows of the edges $e_1,\ldots,e_k$ point from $u$ to $v$, and the
arrows of the edges $e_{k+1},\ldots,e_n$ point from $v$ to $u$. If a
$0$-$1$ flow $f$ of $(B_n,\varepsilon_k)$ has value 1 on exact $j$
edges of $e_1,\ldots,e_k$, then $f$ must have value $1$ on exact $j$
edges of $e_{k+1},\ldots,e_n$. Thus by Lemma~\ref{EER},
$\bar\varphi_{\varepsilon_k}(B_n,1)$ is the number of 0-1 flows of
$(B_n,\varepsilon_k)$, and is given by
\[
\bar\varphi_{\varepsilon_k}(B_n,1)=\sum_{j=0}^{\min(k,n-k)}
\Bigl({k\atop j}\Bigr)\Bigl({n-k\atop j}\Bigr).
\]
For the case $n=4$, we see that
$\bar\varphi_{\varepsilon_0}(B_4,1)=1$,
$\bar\varphi_{\varepsilon_1}(B_4,1)=4$,
$\bar\varphi_{\varepsilon_2}(B_4,1)=6$,
$\bar\varphi_{\varepsilon_3}(B_4,1)=4$, and
$\bar\varphi_{\varepsilon_4}(B_4,1)=1$. According to
Theorem~\ref{Modular-Flow-Thm}, $\bar\varphi_{\varepsilon_k}(B_4,1)$
counts the number of orientations of $B_4$ that are Eulerian
equivalent to $\varepsilon_k$. Indeed, there are $4$ orientations
Eulerian equivalent to $\varepsilon_1$ and $\varepsilon_3$
respectively, and $6$ orientations Eulerian equivalent to
$\varepsilon_2$. These orientations are listed in Figure~\ref{B4}.
However, the Eulerian equivalence classes for $\varepsilon_0$ and
$\varepsilon_4$ are singletons; the integral and modular flow
polynomials with respect to these orientations are the zero
polynomial. We list the two orientations as follows:
\begin{figure}[h]
\centering
\includegraphics[width=50mm]{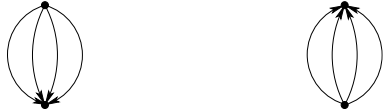} \caption{The other
two Eulerian equivalence classes of orientations of $B_4$ with
directed cut. \label{B4d}}
\end{figure}
\end{exmp}

\noindent{\bf Remark.} The coefficients of the integral flow
polynomial $\varphi_{\mathbbm z}(G,t)$ are not necessarily integers
as shown in Example~\ref{BNUV}. The combinatorial interpretation on
the coefficients of $\varphi_{\mathbbm z}(G,t)$ is particularly
wanted. \vspace{1ex}

\noindent{\bf Acknowledgement.} We thank the referees, in particular
the second referee, for carefully reading the manuscript and
offering several valuable comments.

\bibliographystyle{amsalpha}

\end{document}